\documentclass[12pt]{amsart}
\usepackage[margin=1in]{geometry}

\usepackage{amsmath,amsthm,amssymb,array,enumerate,parskip,graphicx,tabularx}
\usepackage[all]{xy}

\usepackage{xcolor} 
\definecolor{dark-red}{rgb}{0.5,0.15,0.15}
\definecolor{dark-blue}{rgb}{0.15,0.15,0.6}
\definecolor{dark-green}{rgb}{0.15,0.6,0.15}
\usepackage{comment}

\usepackage{amssymb}    
\usepackage{amsmath}    
\usepackage{amsthm}     
\usepackage{amssymb}
\usepackage{mathrsfs}

\numberwithin{equation}{section}

\newtheorem{theorem}{Theorem}[section]
\newtheorem{lemma}[theorem]{Lemma}
\newtheorem{proposition}[theorem]{Proposition}
\newtheorem{corollary}[theorem]{Corollary}

\newenvironment{prf}[1]{\trivlist
\item[\hskip
\labelsep{\it #1.\hspace*{.3em}}]}{
\endtrivlist}

\newtheorem{predefinition}[theorem]{Definition}

\newtheorem{preremark}[theorem]{Remark}

\newenvironment{remark}{\begin{preremark}\rm}{\end{preremark}}
\newtheorem{prenotation}[theorem]{Notation}
\newenvironment{notation}{\begin{prenotation}\rm}{\end{prenotation}}
\newtheorem{preexample}[theorem]{Example}
\newenvironment{example}{\begin{preexample}\rm}{\end{preexample}}
\newtheorem{preclaim}[theorem]{Claim}

\newtheorem{prequestion}[theorem]{Question}

 \makeatletter
\def\emppsubsection{\@startsection{subsection}{2}{\z@}{-3.25ex plus -1ex minus -.2ex}{-1em}{\bf}}

\makeatother

\title{Cohomology groups of Fermat curves
via ray class fields of cyclotomic fields}
\author{Rachel Davis}
\address{University of Wisconsin-Madison}
\email{rachel.davis@wisc.edu}
\author{Rachel Pries}
\address{Colorado State University}
\email{pries@math.colostate.edu}

\thanks{We would like to thank Vesna Stojanoska and Kirsten Wickelgren, for their help with Lemmas \ref{Lstabilize} and 
\ref{Lchangesection} and, more generally, for their ideas in developing this project and 
a wonderful experience collaborating together.  We would like to thank Gras and Maire for suggesting Proposition \ref{Prational} and Nguyen-Quang-Do for helpful comments.
We would like to thank AIM for support for this project through a Square collaboration grant.
Pries was partially supported by NSF grants DMS-15-02227 and DMS-19-01819. 
}

\begin{document}

\maketitle

\begin{abstract}
The absolute Galois group of the cyclotomic field $K={\mathbb Q}(\zeta_p)$
acts on the \'etale homology of the Fermat curve $X$ of exponent $p$.
We study a Galois cohomology group
which is valuable for measuring an obstruction for $K$-rational points on $X$.
We analyze a $2$-nilpotent extension of $K$ which contains the
information needed for measuring this obstruction.
We determine a large subquotient of this Galois cohomology group which arises from Heisenberg extensions of $K$.
For $p=3$, we perform a Magma computation with ray class fields, 
group cohomology, and Galois cohomology which determines it completely.\\
MSC10: primary 11D41, 11R18, 11R34, 11R37, 11Y40; secondary 13A50, 14F20, 20D15, 20J06, 55S35.
\end{abstract}




\section{Introduction}

Let $p$ be an odd prime and let $r=(p-1)/2$.
Consider the cyclotomic field $K=\mathbb{Q}(\zeta_p)$.
Let $Q={\rm Gal}(L/K)$ where $L$ is the splitting field of the polynomial $1-(1-x^p)^p$.
Then $Q$ is an elementary abelian $p$-group.  
For $p$ satisfying Vandiver's conjecture, the rank of $Q$ is $r+1$ \cite[Proposition~3.6]{WINF:birs}.

Let $E$ be the maximal elementary abelian $p$-group extension of $L$ ramified only over $p$.
The field $E$ is contained in a ray class field of $L$.
Let $G={\rm Gal}(E/K)$.  Then, letting $N={\rm Gal}(E/L)$, there is a short exact sequence
\begin{equation} \label{shortexact}
1 \to N \to G \to Q \to 1.
\end{equation}
If $p$ is a regular prime, we prove that ${\rm dim}_{{\mathbb F}_p}(N) = 1 + p^{r+1}(p-1)/2$ in Proposition \ref{Prational}.

There is an element $\omega \in \rm{H}^2(Q, N)$ which classifies the extension \eqref{shortexact}
and determines the isomorphism class of the group $G$. 
After choosing generators for $Q$ and a splitting $s: Q \to G$, 
then $\omega$ is determined by certain elements $a_i, c_{j,k} \in N$ for $0 \leq i \leq r$ and $0 \leq j < k \leq r$, 
see Section \ref{Sclass}. 

Suppose $M$ is a $G$-module on which $N$ acts trivially.
The inflation-restriction exact sequence yields a short exact sequence:
\begin{equation} \label{bigexact}
0 \to \rm{H}^1(Q, M) \to \rm{H}^1(G, M) \to {\rm Ker}(d_2) \to 0, 
\end{equation}
where $d_2:\rm{H}^1(N, M)^Q \to \rm{H}^2(Q, M)$ is the transgression map, which depends on 
$\omega$ as in \cite[3.7]{Koch} and \cite[1.6.6, 2.4.3]{Neukirch-Schmidt-Wingberg}.
In \cite[Theorem 6.11]{WINF:Baction}, given $\phi \in {\rm{H}om}(N, M)^Q$, 
we prove that the class of $\phi$ is in ${\rm Ker}(d_2)$ if and only if 
the values of $\phi$ on $a_i$ and $c_{j,k}$ satisfy certain algebraic properties, see Theorem~\ref{Tkerd2}.

This paper is about the Galois cohomology of the homology of Fermat curves.
The Fermat curve $X$ of exponent $p$ is the smooth curve in ${\mathbb P}^2$ with equation $x^p+y^p=z^p$.
Let $\rm{H}_1(X)$ denote the \'etale homology of $X$.
Anderson proves that $N$ acts trivially on $\rm{H}_1(X; \mathbb{Z}/p \mathbb{Z})$, \cite[Section~10.5]{Anderson}. 
In this paper, we study the Galois cohomology group $\rm{H}^1(G, \rm{H}_1(X; \mathbb{Z}/p \mathbb{Z}))$.
More generally, we study $\rm{H}^1(G, M)$ for subquotients $M$ of the relative homology $\rm{H}_1(U, Y; \mathbb{Z}/p \mathbb{Z})$ where 
$U$ is the affine curve $x^p+y^p=1$ and $Y$ is the set of $2p$ cusps where $xy=0$.
The motivation for studying this Galois cohomology group and ${\rm Ker}(d_2)$
is in Section~\ref{Smotivation}.

The main theme of the paper is that Galois extensions of the cyclotomic field $K$ 
determine information about the kernel of the transgression map ${\rm Ker}(d_2)$.
In Section \ref{Srayfield}, under the condition that $\mathop{\rm Cl}(L)[p]$ is trivial, we analyze how 
$N$, $G$, $\omega$, and ${\rm Ker}(d_2)$ are determined from certain ray class field extensions. 

Suppose $p$ is an odd prime satisfying Vandiver's Conjecture.
The main result of the paper is Theorem \ref{PHeisKer}, in which we determine the subspace of ${\rm Ker}(d_2)$ 
arising from Heisenberg extensions of $K$. 
In Section \ref{Sheisenbergbig}, we calculate the dimension of this subspace for some small $p$.
More generally, in Section \ref{Ssec4}, we consider natural subextensions $\bar{E}/L$ of $E/L$, 
which lead to quotients $\bar{N}$ of $N$ for which we can analyze 
$\rm{H}^1(\bar{N}, M)^Q$, and thus determine a subspace of ${\rm Ker}(d_2)$.    
In Section~\ref{Sabelian}, 
we determine the subspaces of ${\rm Ker}(d_2)$ arising from  
ray class, cyclotomic, and Kummer extensions of $K$; the latter two are trivial unless $p=3$.

In Section \ref{Supperbound}, when $p=3$, we perform an extensive MAGMA calculation to determine 
$N$, $G$, $\omega$, and ${\rm Ker}(d_2)$; in particular, we determine the dimension of ${\rm Ker}(d_2)$ in Corollary \ref{Cfinish3}.
The results in Sections \ref{StransFermat}, \ref{Sheisenbergbig} and \ref{Supperbound} are possible because 
we have explicit knowledge about the action of $Q={\rm Gal}(L/K)$
on $M$ from \cite{WINF:Baction}.

\subsection{Motivation} \label{Smotivation}

The motivation to study the Galois cohomology group $\rm{H}^1(G, \rm{H}_1(X; \mathbb{Z}/p \mathbb{Z}))$ arises from the Kummer map. 
Let $b=[0:1:0]$ be a base point of $X$.
Let $\pi=\pi_1(X_{\bar{K}}, b)$ denote the geometric \'etale fundamental group of $X$ based at $b$.
Consider the lower central series:
\[\pi = [\pi]_1 \supseteq [\pi]_2 \supseteq \ldots \supseteq  [\pi]_n \supseteq \ldots .\] 

Let $G_K$ be the absolute Galois group of $K$. 
For a $K$-rational point $\eta$ of $X$, let $\gamma$ be a path in $X({\mathbb C})$
from $b$ to $\eta$.  The Kummer map 
\begin{equation} \label{Ekummer}
\kappa: X(K) \to \rm{H}^1(G_K, \pi_1(X))
\end{equation} is defined
by $\kappa(\eta) = [\sigma \mapsto \gamma^{-1} \sigma(\gamma)]$ for $\sigma \in G_K$.

The \'etale homology $\rm{H}_1(X)=\pi/[\pi]_2$ is the maximal abelian quotient of $\pi$.
From \eqref{Ekummer}, we obtain a map $\kappa: X(K) \to \rm{H}^1(G_K, \rm{H}_1(X) \otimes \mathbb{Z}_p)$, which is injective.
Let $G_{K,S}$ be the Galois group of the maximal pro-$p$ extension of $K$ ramified only over 
$S=\left\{ \nu \right\}$ where $\nu=\langle 1 - \zeta_p \rangle$.
Since the Fermat curve $X$ has good reduction away from $p$ and $K$ has no infinite primes,
$\kappa$ factors through $\kappa: X(K) \to \rm{H}^1(G_{K,S}, \rm{H}_1(X) \otimes \mathbb{Z}_p)$.

Using work of Schmidt and Wingberg \cite{SW92}, 
Ellenberg \cite{Ellenberg_2_nil_quot_pi} defines a series of obstructions 
to a $K$-rational point of the Jacobian of $X$ lying in the image of the Abel-Jacobi map. 
Namely, via the Kummer map, $X(K)$ and ${\rm Jac}(X)(K)$ can be viewed as subsets of 
$\rm{H}^1(G_{K,S}, \rm{H}_1(X)\otimes \mathbb{Z}_p)$. 
Let $\delta_2$ denote the first of these obstructions; it was also studied by Zarkhin \cite{zarhin}.
By \cite[Proposition~3.2]{SW92}, the map $\delta_2$ also factors through $G_{K,S}$
and has the form
\begin{equation} \label{Edelta}
\delta_2:  \rm{H}^1(G_{K,S}, \rm{H}_1(X)\otimes \mathbb{Z}_p) \to \rm{H}^2(G_{K,S}, ([\pi]_2/[\pi]_3) \otimes \mathbb{Z}_p);
\end{equation}
it is the coboundary map associated to the $p$-part of the exact sequence 
\[ 1 \to [\pi]_2/[\pi]_3  \to \pi/[\pi]_3 \to \pi/[\pi]_2 \to 1.\]

Thus, a complete understanding of $\rm{H}^1(G_{K,S}, \rm{H}_1(X) \otimes \mathbb{Z}_p)$, together with the map $\delta_2$,
provides important information about $K$-rational points on $X$.
In this paper, by taking the homology with coefficients in $\mathbb{Z}/p \mathbb{Z}$, we consider the ``first level'' of the cohomology. 
In Section \ref{Skummerobstruction}, we show that we can replace $G_{K,S}$ by $G$ for this first level 
and we use a spectral sequence to produce the exact sequence \eqref{bigexact}.

The goal of this paper is to analyze the quotient ${\rm Ker}(d_2)$ of $\rm{H}^1(G, \rm{H}_1(X; \mathbb{Z}/p\mathbb{Z}))$.
This material will be needed in future work, where 
we analyze the kernel $\rm{H}^1(Q, M)$ of \eqref{bigexact} and compute the obstruction map $\delta_2$.

\section{Background}

\subsection{Field theory} \label{Sfield}

Let $p$ be an odd prime and let $r=(p-1)/2$.
Let $\zeta_p$ be a primitive $p$th root of unity.  
Let $K=\mathbb{Q}(\zeta_p)$.
By \cite[Lemmas 1.3, 1.4]{washingtonbook}, 
$K$ is ramified only above $p$ and $\langle p \rangle =\nu^{p-1}$ 
where $\nu = \langle 1 - \zeta_p \rangle$ is the unique prime above $p$; also 
$\nu = \langle 1 - \zeta_p^i \rangle$ for $1 \leq i \leq p-1$. 

Let $L$ be the splitting field of $1-(1-x^p)^p$.  Note that $\zeta_p \in L$.
Also $L$ contains the $p$th roots of $t_0=\zeta_{p}$ and 
$t_i= 1-\zeta_p^{-i}$ for $1 \leq i \leq r$.  Let $K_0=K(\zeta_{p^2})$ and $K_i=K(\sqrt[p]{t_i})$.
By \cite[Lemma 3.3]{WINF:birs}, $L$ is the compositum of $K_0$ and $K_i$ for $1 \leq i \leq r$.
By \cite[Lemma 3.3]{WINF:birs}, $L/K$ is only ramified at $\nu$.  
So $L$ is contained in the maximal elementary abelian $p$-extension of $K$ ramified only over the prime above $p$.

\subsection{Galois groups} \label{Sgalois}

Let $Q={\rm Gal}(L/K)$.
We assume throughout that $p$ satisfies Vandiver's Conjecture, namely that $p$ does not divide the order $h^+$ of the class group of $\mathbb{Q}(\zeta_p+\zeta_p^{-1})$; 
this is true for all $p$ less than 163 million and all regular primes. 
By \cite[Proposition 3.6]{WINF:birs}, this implies that $K_0, \ldots, K_r$ are linearly disjoint over $K$ 
and so $Q \simeq (\mathbb{Z}/p \mathbb{Z})^{r+1}$, where $r=(p-1)/2$. 
In particular, the degree $d={\rm deg}(L/\mathbb{Q})$ satisfies $d=(p-1)p^{r+1}$.  

Note that $Q \simeq \times_{i=0}^r {\rm Gal}(K_i/K)$.  
We choose an explicit basis $\{\tau_0, \ldots, \tau_r\}$ for $Q$ as follows. 
For $0 \leq i \leq r$, let $\tau_i \in Q$ be such that $\tau_i(\sqrt[p]{t_i})=\zeta_p \sqrt[p]{t_i}$ and $\tau_i(\sqrt[p]{t_j}) = \sqrt[p]{t_j}$ for $i \not = j$.
Let $\tau_i$ also denote the image of $\tau_i$ in ${\rm Gal}(K_i/K)$.

Let $G_K$ (resp.\ $G_L$) denote the absolute Galois group of $K$ (resp.\ $L$).

\subsection{A 2-nilpotent extension of $K$} \label{Snil}

Let $E$ be the maximal elementary abelian $p$-group extension of $L$ which is ramified only above $p$.
Then $E/K$ is Galois.
Let $N={\rm Gal}(E/L)$ and $G={\rm Gal}(E/K)$.
By definition, $N$ is an elementary abelian $p$-group.
As in \eqref{shortexact}, there is a short exact sequence  
$1 \to N \to  G \to Q \to 1$.

\subsection{Class groups}

Let $\mathop{\rm Cl}(L)$ denote the class group of $L$.
Let $\mathop{\rm Cl}^S(L)$ be the quotient of $\mathop{\rm Cl}(L)$ by the subgroup of classes of ideals generated by the primes of a set $S$.
Let $\mathop{\rm Cl}(L)[p]$ and $\mathop{\rm Cl}^S(L)[p]$ denote the $p$-Sylow subgroups of these.

For a number field $F$, let $r_2(F)$ denote the number of complex places of $F$.
Let $G_{F, p}$ be the Galois group of the maximal pro-$p$ extension of $F$ ramified only above the primes above $p$,
see \cite[Section 11.1]{Koch}.  
Since $L$ is totally complex, $r_2(L)=d/2$.
This implies that no infinite places are ramified in finite extensions of $L$ and that restricted and ordinary 
class groups are equal.

There is a short exact sequence 
\begin{equation} \label{Eexactresramlarge}
1 \to G_{L, p} \to G_{K, p} \to Q \to 1.
\end{equation}

For a finitely generated $p$-group $\Gamma$, let $\Phi(\Gamma)$ denote its Frattini subgroup, 
namely the closed characteristic subgroup of $\Gamma$ topologically generated by $p$th powers and commutators. 
We write $\Phi(\Gamma)=\Gamma^p[\Gamma,\Gamma]$.
The Frattini quotient $\Gamma'=\Gamma/\Phi(\Gamma)$ is an elementary abelian $p$-group.
By Burnside's basis theorem, $\dim_{\mathbb{F}_p}(\Gamma') = \dim_{\mathbb{F}_p}(\Gamma)$. 

By definition, $N$ is the Frattini quotient $G_{L,p}'$ of $G_{L, p}$.
The group $G$ is the pushout of $G_{K,p}$ and $N$ with respect to the inclusion $G_{L,p} \to G_{K,p}$ 
and the quotient map $G_{L, p} \to N$.

\subsection{The Fermat curve} \label{SFermat}

The Fermat curve of exponent $p$ is the smooth projective curve $X \subset {\mathbb P}^2$ given by the 
equation $x^p+y^p=z^p$.  The open affine $U \subset X$ given by $z \not = 0$ has affine equation $x^p+y^p=1$.
Let $Y \subset U$ denote the closed subscheme of $2p$ points with $xy=0$. 

The curve $X$ has good reduction away from $p$.
Thus $U/K$ has good reduction except at $\nu$. 

The group $\mu_p \times \mu_p$ acts on $X$, and this action stabilizes $U$ and $Y$.
Let $\epsilon_0, \epsilon_1$ be the generators of $\mu_p \times \mu_p$ which act by
$\epsilon_0(x,y) = (\zeta_p x, y)$ and $\epsilon_1(x,y) = (x, \zeta_p y)$.
Consider the group ring $\Lambda_1= (\mathbb{Z}/p \mathbb{Z})[\mu_p \times \mu_p]$.

Let $y_i=\epsilon_i-1$.  Then $y_i^p=0$ and $\Lambda_1 \simeq (\mathbb{Z}/p \mathbb{Z})[y_0, y_1]/\langle y_0^p, y_1^p \rangle$.
Consider the augmentation ideal $\langle (1-\epsilon_0)(1-\epsilon_1) \rangle = \langle y_0y_1 \rangle$ of $\Lambda_1$.

\subsection{Homology} \label{Shomology}

We consider \'etale homology groups with coefficients in $\mathbb{Z}/p \mathbb{Z}$.
Let $M=\rm{H}_1(U, Y)=\rm{H}_1(U, Y; \mathbb{Z}/p \mathbb{Z})$ denote the homology of $U$ relative to $Y$.

By \cite[Theorem 6]{Anderson}, $M$ is a free rank one $\Lambda_1$-module, with generator denoted $\beta$.
Under the identification of $M$ with $\Lambda_1$, the homology $\rm{H}_1(U) = \rm{H}_1(U; \mathbb{Z}/p \mathbb{Z})$ identifies with the augmentation ideal 
$\langle y_0 y_1 \rangle \subset \Lambda_1$ \cite[Proposition 6.2]{WINF:birs}.
Furthermore, $\rm{H}_1(X) = \rm{H}_1(X; \mathbb{Z}/p \mathbb{Z})$ is the quotient of $\rm{H}_1(U)$ by ${\rm Stab}(\epsilon_0 \epsilon_1)$ \cite[Proposition 6.3]{WINF:birs}.

In addition, $M$ is a $p$-torsion $G_K$-module.
By \cite[Section 10.5]{Anderson}, the action of $G_K$ on $M$ factors through $L$. 
This implies that $G_L$ and $N$ act trivially on $M$.
This means that the action of $G_K$ on $M$ is determined by the action of $Q={\rm Gal}(L/K)$ on $M$.
Write $\beta$ for the chosen generator of $M$
and let $B_{i} =B_{\tau_i} \in \Lambda_1$ denote the element such that $\tau_i \cdot \beta= B_{i} \beta$.
By \cite[9.6, 10.5.2]{Anderson}, $B_i-1 \in \langle y_0y_1 \rangle$. 

\subsection{The transgression map} \label{Skummerobstruction}

Let $M$ be a $p$-torsion $G_{K,p}$-module. 
From \eqref{Eexactresramlarge} 
and the Lyndon-Hochschild-Serre spectral sequence 
\[\rm{H}^i(Q, \rm{H}^j(G_{L,p} , M)) \Rightarrow \rm{H}^{i+j} (G_{K,p} , M),\] 
we obtain the exact sequence 
\begin{equation} \label{inflationrestriction}
0 \to \rm{H}^1(Q, M^{G_{L,p}})  \to \rm{H}^1 (G_{K,p} , M) \to  \rm{H}^1(G_{L,p}, M)^{Q} \stackrel{d_2}{\rightarrow} \rm{H}^2(Q, M^{G_{L,p}}).
\end{equation}

Here $d_2$ is called the transgression map.
Suppose $G_{L,p}$ acts trivially on $M$, 
then $\rm{H}^i(Q, M^{G_{L,p}})=\rm{H}^i(Q, M)$ for $i=1,2$.

Since $N$ is the Frattini quotient of $G_{L,p}$ and since 
$M$ is a finite dimensional vector space over $\mathbb{F}_p$, 
there is an isomorphism $\rm{H}^1(G_{L,p}, M)^{Q} \simeq \rm{H}^1(N, M)^Q$.
Since $G$ is a quotient of $G_{K,p}$, there is an injection 
$\rm{H}^1(G, M) \to \rm{H}^1(G_{K,p}, M)$.
By the short five lemma, $\rm{H}^1 (G_{K,p} , M) \simeq \rm{H}^1 (G, M)$.
With some abuse of notation, we write $d_2: \rm{H}^1(N, M)^Q \to \rm{H}^2(Q, M)$.
Then there is an exact sequence, as in \eqref{bigexact},
\[0 \to \rm{H}^1(Q, M) \to \rm{H}^1(G, M) \to {\rm Ker}(d_2) \to 0.\]

Since $N$ acts trivially on $M$, an element $\phi \in \rm{H}^1(N,M)^Q$ is uniquely determined by
 a $Q$-invariant homomorphism $\phi:N \to M$.
To compute $d_2$, we consider a $2$-cocycle $\tilde{\omega}: Q \times Q \to N$ for the element $\omega \in \rm{H}^2(Q, N)$ classifying
the extension \eqref{shortexact}.
By \cite[1.6.6, 2.4.3]{Neukirch-Schmidt-Wingberg} and \cite[3.7 (3.9) and (3.10)]{Koch} (see \cite[Proposition 6.1]{WINF:Baction}), 
$d_2(\phi) = - \phi \circ \tilde{\omega}$. 

\section{Ray class fields and the classifying element} \label{Srayfield}

Suppose that $p$ is an odd prime satisfying Vandiver's Conjecture.
From Sections \ref{Sfield} and \ref{Sgalois}, 
recall that $K=\mathbb{Q}(\zeta_p)$, $L$ is the splitting field of the polynomial $1-(1-x^p)^p$,
and $Q={\rm Gal}(L/K)$ is an elementary abelian $p$-group of rank $r+1$, where $r= (p-1)/2$.

From Section \ref{Snil}, $E$ is the maximal elementary abelian $p$-group extension of $L$, 
ramified only above $p$.
Let $G={\rm Gal}(E/K)$ and $N={\rm Gal}(E/L)$.
There is an exact sequence:
\begin{equation} \label{Egaloisexact} 
1 \to N \to G \to Q \to 1.
\end{equation}

In Sections \ref{Sregular}, \ref{ScompareST}, and \ref{Srayclass}, under the condition that $p$ is regular, 
we use class field theory to give results on $N$ and its connection with ray class fields.
Section \ref{Sclass} contains the material needed to classify the extension \eqref{Egaloisexact}.
If $M$ is a $G$-module on which $N$ acts trivially, in Section \ref{Salgebraker}, we give an algebraic description of  
the kernel of the transgression map $d_2:\rm{H}^1(N, M)^Q \to \rm{H}^2(Q, M)$.
We specialize this to the case that $M$ is the relative homology of the Fermat curve 
in Section \ref{StransFermat}.

\subsection{The rank of $N$ when $p$ is regular} \label{Sregular}

Using the topic of $p$-rationality, 
we determine more information about $L$ and $N$ when $p$ is a regular prime.
A good reference for $p$-rationality is \cite{MovDo} or \cite[IV, Section 3]{Grasbook}.
A number field $M$ is $p$-rational when $G_{M,p}$ is a free pro-$p$ group of rank $1+r_2(M)$.
See other equivalent definitions in \cite[IV, Remark~3.4.5, Theorem~3.5]{Grasbook}. 

\begin{proposition} \label{Prational}
If $p$ is a regular prime, then $L$ is $p$-rational and ${\rm dim}_{{\mathbb F}_p}(N) = 1 + d/2$.
Also, there is a unique prime $\mathfrak{p}$ of $L$ above $p$ and
$\mathop{\rm Cl}^{\{p\}}(L)[p]$ is trivial.
\end{proposition}

(Recall that $\mathop{\rm Cl}^{\{p\}}(L)[p]$ is the $p$-Sylow subgroup of the quotient $\mathop{\rm Cl}^{\{p\}}(L)$ of $\mathop{\rm Cl}(L)$ 
by the subgroup of classes of ideals generated by $\mathfrak{p}$.)

\begin{proof}
If $p$ is a regular prime, then $K$ is $p$-rational by \cite[Proposition 3, Example page 166]{Movthesis}.
Since $L/K$ is a Galois $p$-extension unramified outside $p$, $L/K$ is $p$-primitively ramified, 
e.g., \cite[page 330]{GrasK2}.
Then $L$ is $p$-rational by \cite[Theorem 3]{Movthesis}.  
Since $L$ is p-rational, 
then $G_{L,p}$ is a free pro-$p$ group of rank $1+r_2(L)$ where $r_2(L)=d/2$.
Thus $N \simeq (\mathbb{Z}/p \mathbb{Z})^{1+d/2}$.
The other claims follow from \cite[Proposition 3]{Movthesis} or \cite[Theorem~4.1(iva)]{GrasJaulent}. 
\end{proof}

We remark that it might be possible to extend the results to the case that $p$ satisfies 
Vandiver's conjecture, using the techniques of \cite{GrasVandiver, GrasIncomplete}.

\subsection{Local and global units} \label{ScompareST}

Let $\mathcal{O}_L$ denote the maximal order of $L$. 
Let $\mathcal{O}_L^{\times}/p = \mathcal{O}_L^{\times}/( \mathcal{O}_L^{\times})^p$
denote the global units modulo $p$. 

By Proposition \ref{Prational}, if $p$ is a regular prime, then 
there is a unique prime $\mathfrak{p}$ of $L$ above $p$.

\begin{corollary} \label{rrbound}
If $p$ is a regular prime, 
then the $p$-rank of $(\mathcal{O}_L/\mathfrak{p}^{n+1})^{\times}$ is $d+1$ if $n \geq p^{r+2}$.
\end{corollary}

\begin{proof}
Let $e$ (resp.\ $f$) denote the ramification index (resp.\ residue degree) of $\mathfrak{p}$ over $\mathbb{Q}$.
Set $e_1= \lfloor e/(p-1) \rfloor$. 
By \cite[Theorem 1, page 45]{Nakagoshi} or \cite[Theorem 1, page 31]{Sengun}, since $\zeta_p \in L$,
the $p$-rank of $(\mathcal{O}_L/\mathfrak{p}^{n+1})^{\times}$ is $ef+1$ if $n \geq e+e_1$.
The result follows since $e=d =(p-1)p^{r+1}$, $f=1$, and $e_1=p^{r+1}$. 
\end{proof}  

Let $\mathcal{O}_{\mathfrak{p}}^{\times}/p = \mathcal{O}_{\mathfrak{p}}^{\times}/( \mathcal{O}_{\mathfrak{p}}^{\times})^p$ denote the local units modulo $p$.
Both $\mathcal{O}_L^{\times}/p$ and $\mathcal{O}_{\mathfrak{p}}^{\times}/p$ are $\mathbb{F}_p$-vector spaces 
with an action of $Q={\rm Gal}(L/K)$.
Let $*$ denote the dual. 

\begin{proposition} \label{PclassGL}
If $\mathop{\rm Cl}(L)[p]$ is trivial, then there is a $Q$-invariant short exact sequence: 
\begin{equation} \label{Elocalglobal}
0 \to \rm{H}^1(N, \mathbb{F}_p) \to (\mathcal{O}_{\mathfrak{p}}^{\times}/p)^* \stackrel{\varphi_2^*}{\to} (\mathcal{O}_L^{\times}/p)^{*} \to 1.
\end{equation}
\end{proposition}

\begin{proof}
The hypothesis that $\mathop{\rm Cl}(L)[p]$ is trivial implies that $p$ is a regular prime.  
Let $\varphi_2^*$ be the dual to the homomorphism
$\varphi_2: \mathcal{O}_L^{\times}/p \to \mathcal{O}_{\mathfrak{p}}^{\times}/p$
induced from the inclusion $\mathcal{O}_L \to \mathcal{O}_{\mathfrak{p}}$.
By \cite[pages 114-115, Theorem 11.7]{Koch}, there is a $Q$-invariant exact sequence 
\begin{align*}  
0 \to \rm{H}^2(\mathop{\rm Cl}(L)[p], \mathbb{F}_p) \stackrel{\textrm{inf}}{\to} \rm{H}^1(G_{L, p}, \mathbb{F}_p) \to 
(\mathcal{O}_{\mathfrak{p}}^{\times}/p )^* \stackrel{\varphi_2^*}{\to}  \mathcal{O}_L^{\times}/p \to B_{\mathfrak{p}} \to 0.\end{align*} 
Since $\mathop{\rm Cl}(L)[p]$ is trivial, $\rm{H}^2(\mathop{\rm Cl}(L)[p], \mathbb{F}_p) = 0$.
Also, $\rm{H}^1(G_{L, p}, \mathbb{F}_p) \simeq \rm{H}^1(N, \mathbb{F}_p)$.
By \cite[page~120]{Koch}, ${\rm dim}_{\mathbb{F}_p}(\mathcal{O}_L^{\times}/p)=d/2$ and
${\rm dim}_{\mathbb{F}_p}(\mathcal{O}_{\mathfrak{p}}^{\times}/p)= d +1$.
By Proposition \ref{Prational}, ${\rm dim}_{{\mathbb F}_p}(N) = 1 + d/2$.
The result follows since $B_{\mathfrak{p}}$ is trivial by a dimension count.
\end{proof}

\subsection{Ray class fields} \label{Srayclass}

Let $L_{\mathbf{m}}$ (resp.\ $\mathop{\rm Cl}_{\mathbf{m}}(L)$) denote the ray class field (resp.\ group) of $L$ of modulus $\mathbf{m}$.
Every extension of $L$ has a conductor, a minimal admissible modulus, which is only divisible by the ramified primes.
Thus the field $E$ is contained in the ray class field $L_{\mathfrak{p}^i}$, for $i$ sufficiently large.
Since $L$ is totally complex, the narrow ray class group is the same as the ray class group, 
e.g., \cite[page 368]{Narkiewicz}.
Let $\left( \mathcal{O}_L/\mathbf{m} \right)^{\times}$ denote the units mod $\mathbf{m}$ of $L$.

\begin{lemma} \label{Lstabilize}
If $\mathop{\rm Cl}(L)[p]$ is trivial, then the $p$-ranks of $\mathop{\rm Cl}_{\mathfrak{p}^i}(L)$ and $(\mathcal{O}_L/\mathfrak{p}^i)^{\times}$ stabilize at the same index $i$.
\end{lemma}

\begin{proof}
Consider the exact sequence \cite[3.2.3]{Cohen}:
\begin{equation} \label{Ecohenray}
U(L)[p] \stackrel{\rm{H}o_\mathbf{m}}{\to} ( \mathcal{O}_L / \mathbf{m} )^{\times}[p]
\stackrel{\psi}{\to}   {\mathop{\rm Cl}}_{\mathbf{m}}(L)[p]   \stackrel{\phi}{\to}  \mathop{\rm Cl}(L)[p] \to 0.
\end{equation}
If $\mathop{\rm Cl}(L)[p]$ is trivial, then ${\mathop{\rm Cl}}_{\mathbf{m}}(L)[p] =  ( \mathcal{O}_L / \mathbf{m} )^{\times}[p]/(\rm{H}o_{\mathbf{m}}(U(L)[p])$ by \eqref{Ecohenray}.
Consider \eqref{Ecohenray} for the moduli $\mathbf{m}=\mathfrak{p}^i$ and $\mathfrak{p}^{i+1}$.
There is a commutative diagram
\[\xymatrix{ U(L)[p] \ar[r]^{\rm{H}o_{\mathfrak{p}^{i+1}}} \ar@{=}[d] & (\mathcal{O}_L/\mathfrak{p}^{i+1})^{\times}[p] \ar[d] \\
U(L)[p] \ar[r]^{\rm{H}o_{\mathfrak{p}^i}} & (\mathcal{O}_L/\mathfrak{p}^i)^{\times}[p].} \]
Consider the surjection of the Frattini quotients
$(\mathcal{O}_L/\mathfrak{p}^{i+1})^{\times}[p]'  \to (\mathcal{O}_L/\mathfrak{p}^i)^{\times}[p]'$.
The $p$-rank of $(\mathcal{O}_L/\mathfrak{p}^i)^{\times}$ stabilizes at index $i$ if and only if $i$ is the first value such that this
surjection is an isomorphism. 
This is equivalent to the equality $\dim_{\mathbb{F}_p}({\rm Im}(\rho_{\mathfrak{p}^{i+1}}))=\dim_{\mathbb{F}_p}({\rm Im}(\rho_{\mathfrak{p}^i}))$ or the fact that the $p$-rank of $\mathop{\rm Cl}_{\mathfrak{p}^i}(L)[p]$ stabilizes at index $i$.
\end{proof}

\subsection{Classifying the extension} \label{Sclass}

Let $N=\operatorname{Gal}(E/L)$, $G=\operatorname{Gal}(E/K)$, and $Q=\operatorname{Gal}(L/K)$.  
Let $\omega$ be the element of $\rm{H}^2(Q, N)$ classifying
\eqref{Egaloisexact} $1 \to N \to G \to Q \to 1$.
Since $Q$ and $N$ are both elementary abelian $p$-groups, 
the structure of $\rm{H}^2(Q,N)$ can be computed abstractly.  However, the precise identification 
of $\omega \in \rm{H}^2(Q, N)$ depends intrinsically on the structure of $G$. 

For example, by Lemma \ref{LdominateH}, $G$ surjects onto a Heisenberg group.
Then \eqref{Egaloisexact} is not split because 
the short exact sequence for the Heisenberg group has no splitting. 
This implies that $\omega$ is non-trivial (i.e., $G$ is not a semi-direct product).

Consider a section $s:Q \to G$ of the extension \eqref{Egaloisexact}.
Without loss of generality, we assume that $s(1)=1$ and
\begin{equation} \label{Esectionextend}
s(\tau_0^{e_0} \cdots \tau_r^{e_r}) = s(\tau_0)^{e_0}\cdots s(\tau_r)^{e_r}, \text{   for } 0\leq e_i \leq p-1.
\end{equation}

Then there is a $2$-cocycle $\tilde{\omega}:Q\times Q \to N$ defined with the formula
\[\tilde{\omega}(q_1,q_2) = s(q_1) s(q_2) s(q_1q_2)^{-1}. \]
The class of $\tilde{\omega}$ in $\rm{H}^2(Q, N)$ is $\omega$;
in particular, it does not depend on the choice of $s$. 

Consider the generators $\tau_i$ with $0 \leq i \leq r$ for 
$Q \simeq (\mathbb{Z}/p\mathbb{Z})^{r+1}$ from Section~\ref{Sgalois}.
For $0\leq i \leq r$, define elements $a_i$ by
\begin{equation}\label{Daaa}
 a_i = s(\tau_i)^p,
\end{equation}
and for $0 \leq j < k \leq r$, define $c_{j,k}$ by
\begin{equation}\label{Dccc}
c_{j,k} = [ s(\tau_k), s(\tau_j) ] = s(\tau_k) s(\tau_j) s(\tau_k)^{-1} s(\tau_j)^{-1} .
\end{equation}
Note that $a_i,c_{j,k} \in N$ since their images in $Q$ are trivial.
These values provide a useful way to classify the extension \eqref{Egaloisexact}
and play a key role in our analysis of ${\rm Ker}(d_2)$, see Theorem~\ref{Tkerd2}.

The difficulty is that not every section $s$ satisfies \eqref{Esectionextend}.
Thus, following \cite[IV, \S 3]{Brown},  
suppose $\omega': Q \times Q \to N$ is another $2$-cocycle representing the class
$\omega \in \rm{H}^2(Q,N)$.  We may choose $\omega'$ such that
$\omega'(q,1)=\omega'(1,q)=1$ for all $q \in Q$.
By \cite[page~92]{Brown}, $\omega'$ determines a unique extension as in \eqref{Egaloisexact},
together with a section $t: Q \to G$ such that $t(1) = 1$.
By \cite[IV \S 3 (3.3)]{Brown}, the correspondence between $t$ and $\omega'$ is described by \begin{equation}\label{section-to-cocycle}\omega'(q_1,q_2) = t(q_1)t(q_2) t(q_1q_2)^{-1}.\end{equation} 
This yields the following description of $G=\operatorname{Gal}(E/K)$
as an abstract group: the elements of $G$ are in bijection with $N \times Q$; 
this bijection takes $(n, q)$ to $n t(q)$. 
The group law is:
$$(n_1,q_1)(n_2,q_2) = (n_1 (q_1 n_2)  \omega(q_1,q_2), q_1 q_2).$$

The section $t$ for $\omega$ might not satisfy the conditions in \eqref{Esectionextend}.
To fix this, we set $s(\tau_i) = t(\sigma) = 0 \times \tau_i$ and extend $s$ to a set-theoretic section 
$s: Q \to G$ using \eqref{Esectionextend}.
Next, we show that the values of $a_i=s(\tau_i)^p$ and $c_{j,k}=[s(\tau_k), s(\tau_j)]$ 
can be computed from $\omega'$. 

\begin{lemma} \label{Lchangesection} With notation as above:
\[a_i = \sum_{\ell = 1}^{p-1} \omega'(\tau_i^\ell, \tau_i) \ {\rm and} \ 
c_{j,k}= \omega'(\tau_k, \tau_j) - \omega'(\tau_j, \tau_k).\]
\end{lemma}

\begin{proof}
First, by \eqref{section-to-cocycle}, $\omega'(\tau_i^\ell, \tau_i) = t(\tau_i^\ell) t(\tau_i) t(\tau_i^{\ell+1})^{-1}$.
Taking the telescoping product yields
\begin{equation} \label{Erone}
\omega'(\tau_i, \tau_i) \omega'(\tau_i^2, \tau_i) \cdots \omega'(\tau_i^{p-1}, \tau_i) = t(\tau_i)^p.
\end{equation}

Second, by \eqref{section-to-cocycle}, 
\begin{equation}\label{ssigmatau}
\omega'(\tau_k, \tau_j)=t(\tau_k) t(\tau_j) t(\tau_k \tau_j)^{-1}.\end{equation} 
By \eqref{section-to-cocycle},
$\omega'(\tau_j, \tau_k) = t(\tau_j) t(\tau_k) t(\tau_j \tau_k)^{-1}$. 
Since $\tau_j \tau_k = \tau_k \tau_j$ in 
$Q$, then $t(\tau_j \tau_k) =t(\tau_k \tau_j)$. So
\begin{equation} \label{ssigmatau2}
\omega'(\tau_j, \tau_k)^{-1} = t(\tau_k \tau_j) t(\tau_k)^{-1} t(\tau_j)^{-1}.
\end{equation}
Multiplying \eqref{ssigmatau} and \eqref{ssigmatau2} yields 
\begin{equation} \label{Ertwo}
\omega'(\tau_k, \tau_j) \omega'(\tau_j, \tau_k)^{-1}= t(\tau_k) t(\tau_j) t(\tau_k)^{-1} t(\tau_j)^{-1} =[t(\tau_k),t(\tau_j)].
\end{equation}

To finish, we replace $t(\tau_i)$ with $s(\tau_i)$ in \eqref{Erone} and \eqref{Ertwo} and rewrite the equations additively.
\end{proof}

\begin{remark}
Lemma \ref{Lchangesection} is a generalization of \cite[Lemma 6.10]{WINF:Baction}.
To see this, note that if the section $t$ does satisfy \eqref{Esectionextend} then 
$\omega'(\tau_i^\ell, \tau_i)=0$ for $0 \leq \ell \leq p-2$.
\end{remark}

\subsection{The transgression map} \label{Salgebraker}

Let $M$ be a $G$-module such that $N$ acts trivially on $M$.
Since the action of $N$ on $M$ is trivial,
an element $\phi \in \rm{H}^1(N, M)$ is uniquely determined by a homomorphism $\phi: N \to M$.

Furthermore, $Q$ acts by conjugation on $N$. 
Then $\phi \in \rm{H}^1(N, M)^Q$ if and only if
the homomorphism $\phi$ is $Q$-invariant, i.e., $\phi(q \cdot n) = q \cdot \phi(n)$
for all $n \in N$.

As in Section \ref{Skummerobstruction}, associated with the exact sequence \eqref{Egaloisexact}, 
there is the transgression map
\begin{equation} \label{Ed2}
d_2: \rm{H}^1(N, M)^{Q} \to \rm{H}^2(Q, M).
\end{equation}

We now give an algebraic description of ${\rm Ker}(d_2)$.
Recall the definitions of $a_i, c_{j,k} \in N$ from \eqref{Daaa} and \eqref{Dccc}. 
Write $N_{\tau_i} = 1 + {\tau_i} + \cdots + {\tau_i}^{p-1}$ for the norm of $\tau_i$.

\begin{theorem} \label{Tkerd2} \cite[Theorem 1.2]{WINF:Baction}.
Let $M$ be a $G$-module such that $N$ acts trivially on $M$.
Suppose $\phi \in \rm{H}^1(N,M)^Q$ is a class represented by a $Q$-invariant homomorphism $\phi: N \to M$.
Then $\phi$ is in the kernel of $d_2$ if and only if there exist $m_0, \ldots, m_r \in M$ such that
\begin{enumerate}
\item $\phi(a_i)= - N_{\tau_i} m_i$ for $0 \leq i \leq r$ and
\item $\phi(c_{j,k})=(1-\tau_k)m_j - (1-\tau_j)m_k$ for $0 \leq j < k \leq r$. 
\end{enumerate}
\end{theorem}

\subsection{The transgression map for Fermat curves} \label{StransFermat}

We now specialize to the Fermat curve setting.  
If $M$ is any subquotient of the relative homology $\rm{H}_1(U,Y; \mathbb{Z}/p \mathbb{Z})$ of the Fermat curve of degree $p$, then 
$M$ is a $G$-module on which $N$ acts trivially.

Recall that $Q={\rm Gal}(L/K)$ is generated by $\{\tau_i \mid 0 \leq i \leq r\}$.
From Section~\ref{Shomology},
$B_i \in \Lambda_1$ is such that $\tau_i \cdot \beta = B_i \beta$.
The norm of $B_{\tau_i}$ is almost always zero, 
which is useful for determining subspaces of the 
kernel of $d_2$ arising from certain extensions of $K$ in the next section.

\begin{theorem} \label{Tnorm}
\cite[Theorem 4.6]{WINF:Baction} 
Suppose $M=\rm{H}_1(U, Y; \mathbb{Z}/p\mathbb{Z})$ is the relative homology of the Fermat curve of exponent $p$.
Let $N_{\tau_i}$ denote the norm of $B_i$ in $\Lambda_1$.
If $p \geq 5$, then $N_{\tau_i} =0$ for $0 \leq i \leq r$;
if $p=3$, then $N_{\tau_1}=0$ but $N_{\tau_0} = y_0^2y_1^2$.
\end{theorem}

\section{Subspaces of the kernel of the transgression map} \label{Ssec4}

The results in this section apply for any odd prime $p$ satisfying Vandiver's Conjecture.
The main theme is that Galois extensions of the cyclotomic field $K=\mathbb{Q}(\zeta_p)$ 
determine subspaces of the kernel ${\rm Ker}(d_2)$ of the transgression map.

We study this when $M$ is a subquotient of the relative homology of the Fermat curve of degree $p$.
In Section \ref{Sabelian}, we determine the subspaces of ${\rm Ker}(d_2)$
arising from ray class, cyclotomic and Kummer field extensions of $K$;
Propositions \ref{Praykerd2}, \ref{Pcyckerd2}, \ref{Pkumkerd2} show that these are frequently trivial. 
The main result is Theorem \ref{PHeisKer} in Section \ref{Sheisclass} in which we determine the subspace of ${\rm Ker}(d_2)$ arising from Heisenberg extensions of $K$.  We compute this subspace with Magma for $p=3,5,7$
using our explicit knowledge of the action of $G_K$ on $M$.

\subsection{Subextensions and subspaces}

Suppose $\bar{E}$ is a subfield of $E$ containing $L$ such that $E/K$ is Galois.  
Let $\bar{G}={\rm Gal}(\bar{E}/K)$ and let 
$\bar{N} = {\rm Gal}(\bar{E}/L)$.
Then $\bar{G}$ is a quotient of $G$ and $\bar{N}$ is a quotient of $N$.
In this situation, there is an exact sequence
\begin{equation}\label{Egaloisexactbar} 
1 \to \bar{N} \to \bar{G} \to Q \to 1.
\end{equation}
An element $\bar{\phi} \in \rm{H}^1(\bar{N}, M)^Q$ is uniquely determined by a
$Q$-invariant homomorphism $\bar{\phi}: \bar{N} \to M$. 

\begin{lemma} \label{Lcentral}
If the conjugation action of $Q$ on $\bar{N}$ is trivial,
then $\rm{H}^1(\bar{N}, M)^Q \simeq (M^Q)^{\rho}$ where $\rho = {\rm dim}_{\mathbb{F}_p}(\bar{N})$.
\end{lemma}

\begin{proof}
This is clear since $\bar{\phi} \in \rm{H}^1(\bar{N}, M)$ is $Q$-invariant if and only if $\bar{\phi}(\bar{N}) \subset M^Q$.
\end{proof}

\begin{lemma} \label{LdecomposeH1}
There is a natural inclusion $\iota: \rm{H}^1(\bar N, M)^Q \hookrightarrow \rm{H}^1(N, M)^Q$.
\end{lemma}

\begin{proof}
This is true because the action of $N$ on $M$ is trivial and $\bar N$ is $Q$-invariant. 
More explicitly, $\iota(\bar{\phi})$ is the composition of the surjective
reduction map $N \to \bar{N}$ with $\bar{\phi}$; if $\bar{\phi}$ is non-trivial, then so is $\iota(\bar{\phi})$.
\end{proof}

Associated with the exact sequence \eqref{Egaloisexactbar}, 
there is a differential map
\begin{equation} \label{Ed2bar}
\bar{d}_2: \rm{H}^1(\bar{N}, M)^{Q} \to \rm{H}^2(Q, M). 
\end{equation}

\begin{lemma} \label{Ld2d2bar}
Then $\bar{d}_2 = d_2 \circ \iota$ and ${\rm Ker}(\bar{d}_2) \subset {\rm Ker}(d_2)$.
\end{lemma}

\begin{proof}
This follows from the description of $\iota(\bar{\phi})$ in the proof of Lemma \ref{LdecomposeH1}.
\end{proof}

For the following types of extensions of $K$, we determine the element 
in $\rm{H}^2(Q, \bar{N})$ classifying the extension \eqref{Egaloisexactbar} and determine the resulting
subspace ${\rm Ker}(\bar{d}_2)$ of ${\rm Ker}(d_2)$:
ray class, cyclotomic, Kummer, Heisenberg, and $U_4$ extensions. 

\subsection{Information from ray class, cyclotomic, and Kummer extensions of $K$} \label{Sabelian}

\subsubsection{The ray class group of $K$}

Let $\Phi(G)=G^p[G,G]$ denote the Frattini subgroup of $G$.
Let $\tilde{L}$ be the fixed field of $E$ over $K$ by $\Phi(G)$.
By definition, $\tilde{G}={\rm Gal}(\tilde{L}/K)$ is the elementary abelian $p$-group $G/\Phi(G)$.
Also $\tilde{L}$ is the maximal elementary abelian $p$-group extension of $K$ ramified only over $\nu$.
Note that $L \subseteq \tilde{L}$.
Let $\rho \in \mathbb{Z}^{\geq 0}$ be such that $\mathrm{deg}(\tilde{L}/L)=p^\rho$.

Note that $\Phi(G) \subset N$ since $Q$ is an elementary abelian $p$-group.
Let $\tilde{N} = N/\Phi(G)$.
By definition, $\tilde{L}/L$ is Galois with group $\tilde{N} = {\rm Gal}(\tilde{L}/L)$.

By Lemma \ref{Ld2d2bar}, ${\rm Ker}(\tilde{d}_{2}) \subset {\rm Ker}(d_2)$, 
where $\tilde{d}_{2}: \rm{H}^1(\tilde{N}, M)^Q \to \rm{H}^2(Q, M)$.

\begin{proposition} \label{Praykerd2}
Suppose $M$ is a $G$-module on which $\tilde{N}$ acts trivially.
If $\mathrm{deg}(\tilde{L}/L)=p^\rho$, then
${\rm Ker}(\tilde{d}_{2}) = \rm{H}^1(\tilde{N}, M)^Q \simeq (M^Q)^\rho$. 
\end{proposition}

\begin{proof}
By Lemma \ref{Lcentral}, $\rm{H}^1(\tilde{N}, M)^Q \simeq (M^Q)^\rho$.
Since $\tilde{G}$ is an elementary abelian $p$-group, 
the images $\tilde{a}_i, \tilde{c}_{j,k} \in \tilde{N}$ of 
$a_i, c_{j,k} \in N$ are all trivial.  
The conditions in Theorem~\ref{Tkerd2} 
for $\phi$ to be in ${\rm Ker}(\tilde{d}_{2})$ are all satisfied, 
by taking $m_i =0$ for $0 \leq i \leq r$.
Thus ${\rm Ker}(\tilde{d}_{2}) = \rm{H}^1(\tilde{N}, M)^Q$.
\end{proof}

\begin{example}
Using Magma \cite{Magma}, we compute that $\rho = 0$ for $p < 37$ and $\rho =1$ for $p=37$, 
with the computation for $p \geq 23$ depending on the generalized Riemann hypothesis.
To see this, consider the ray class group of $K$ for the modulus $\mathfrak{m}=(1-\zeta_p)^i$.
By \cite[Theorem~1]{Nakagoshi} 
or \cite[Theorem~1]{Sengun}, 
the rank of its maximal elementary abelian $p$-group quotient stabilizes 
beyond some index $i$.  Also $i=e+\lfloor \frac{e}{p-1} \rfloor +1$, 
where $e=p-1$ is the ramification index of $\langle 1 - \zeta_p \rangle$ of $K$ above $p$, so $i=p+1$.
Then $\rho = \mathrm{dim}_{\mathbb{F}_p} \mathrm{Cl}_{(1-\zeta_p)^i}(K)[p] - (r+1)$.
\end{example}

\subsubsection{A cyclotomic extension of $K$} \label{Scyclotomic}

Let $v=\zeta_{p^3}$.
Let $w=v^p=\zeta_{p^2}$ and note that $w \in L$.
Let $L^*=L(v)$.
Let $r=(p-1)/2$.

\begin{lemma} \label{Lgaloiscyclotomic}
The extension $L^*/K$ is Galois and is ramified only over $\nu$.
The Galois group $G^*={\rm Gal}(L^*/K)$ is isomorphic to $\mathbb{Z}/{p^2} \mathbb{Z} \times (\mathbb{Z}/p\mathbb{Z})^r$.
\end{lemma}

\begin{proof}
This is because $L^*/K$ is the compositum of the $\mathbb{Z}/{p^2 \mathbb{Z}}$-extension
$K(v)/K$ and the $\mathbb{Z}/p \mathbb{Z}$-extensions $K(\sqrt[p]{t_i})/K$ where $t_i=1-\zeta_p^{-i}$ for $1 \leq i \leq r$. 
These extensions are disjoint and each ramified only over $\nu$.
\end{proof}

By Lemma \ref{Lgaloiscyclotomic}, $L \subset L^* \subset E$.  Then $N^* = {\rm Gal}(L^*/L)$ is a quotient of $N$.
Let $a_i^*, c_{j,k}^*$ denote the images of $a_i, c_{j,k} \in N$ in $N^*={\rm Gal}(L^*/L)$ for $0 \leq  i \leq r$ and $0 \leq j < k \leq r$.

\begin{lemma} \label{Lclasscyc}
With notation as above, $N^*$ is generated by $a_0^*$.
Also, $a_i^*$ is trivial for $1 \leq i \leq r$ and $c_{j,k}^*$ is trivial for $0 \leq j < k \leq r$.
\end{lemma}

\begin{proof}
By Lemma \ref{Lgaloiscyclotomic}, 
$G^*={\rm Gal}(L^*/K)$ is abelian and generated by automorphisms 
$\tau_i^*$ whose image in $Q$ is $\tau_i$ for $0 \leq i \leq r$;
also $\tau_0^*$ has order $p^2$ and $\tau_i^*$ has order $p$ for $1 \leq i \leq r$.

In particular, $\tau_0^* (v)=v^e$ for some exponent $e \in \mathbb{Z}/p^3 \mathbb{Z}$ such that 
$p \nmid e$.
The relation $v^p=w$ implies that $\tau_0^*(w) = w^e$.  Thus $e \equiv p+1 \bmod p^2$.

By definition, $a_0^*= (\tau_0^*)^p(v)/v$.  Now $\tau_0^*(v) = v^{e^p}$.
The condition $e \equiv p+ 1 \bmod p^2$ implies that $e^p \equiv 1+p^2 \bmod p^3$.
Thus $a_0^* = v^{e^p-1} = v^{p^2}= \zeta_p$.  Thus $a_0^*$ is non-trivial and thus generates $N^*$.

By definition, $c_{j,k}^* = \tau_k^* \tau_j^* (\tau_k^*)^{-1}  (\tau_j^*)^{-1} (v)/v$.
Since $G^*$ is abelian, $c_{j,k}^*$ is trivial.

For $1 \leq i \leq r$, by definition, $a_i^* = (\tau_i^*)^p(v)/v$.  
Since $\tau_i^*$ has order $p$, $a_i^*$ is trivial.
\end{proof}

\begin{proposition} \label{PCYC}
If $M$ is a $G$-module on which $N^*$ acts trivially,
then $\rm{H}^1(N^*, M)^Q \simeq M^Q$. 
\end{proposition}

\begin{proof}
This follows from Lemma \ref{Lcentral} and Lemma \ref{Lclasscyc},
with the isomorphism $\rm{H}^1(N^*, M)^Q \simeq M^Q$
identifying $\bar{\phi}^*$ with the value
$\mu^*=\bar{\phi}^*(a_0^*) \in M^Q$.
\end{proof}

By Lemma \ref{Ld2d2bar}, ${\rm Ker}(d_{2}^*) \subset {\rm Ker}(d_2)$ where $d_{2}^*: \rm{H}^1(N^*, M)^Q \to \rm{H}^2(Q, M)$.

\begin{proposition} \label{Pcyckerd2}
Let $X$ be the Fermat curve of degree $p$.
Suppose $M$ is a subquotient of the relative homology $\rm{H}_1(U, Y; \mathbb{Z}/p \mathbb{Z})$.
\begin{enumerate}
\item If $p \geq 5$, then ${\rm Ker}(d_2^*) =0$.  The same is true for $p=3$ when $M=\rm{H}_1(X; \mathbb{Z}/3 \mathbb{Z})$. 
\item If $p =3$ and $M=\rm{H}_1(U; \mathbb{Z}/3 \mathbb{Z})$ or $M=\rm{H}_1(U,Y; \mathbb{Z}/3 \mathbb{Z})$, then
${\rm Ker}(d_2^*)$ has dimension $1$ and is generated by the homomorphism $\phi^*:N^* \to M$
such that $\phi^*(a_0^*)=y_0^2y_1^2$.
\end{enumerate}
\end{proposition}

\begin{proof}
A $Q$-invariant morphism 
$\phi^*: N^* \to M$ is determined by its image on the generator $a_0^*$ of $N^*$.
Theorem \ref{Tkerd2} contains the conditions for $\phi^*$ to be in ${\rm Ker}(d_2^*)$.
By Theorem \ref{Tnorm}, $N_{\tau_i}=0$ for $1 \leq i \leq r$.
Since $c_{j,k}^*=0$, and $a_i^*=0$ for $1 \leq i \leq r$, these conditions are satisfied 
if and only if $\phi^*(a_0^*)= - N_{\tau_0} m_0$ for some $m_0 \in M$.
If $p \geq 5$ then $N_{\tau_0}=0$.  
If $p=3$, then $N_{\tau_0}=y_0^2y_1^2$, which is trivial in $\rm{H}_1(X; \mathbb{Z}/3 \mathbb{Z})$, 
but not in $\rm{H}_1(U; \mathbb{Z}/3 \mathbb{Z})$ or $\rm{H}_1(U,Y; \mathbb{Z}/3 \mathbb{Z})$.
\end{proof}

The theory of higher ramification groups for a ramified prime in an extension of number fields 
can be found in \cite[Chapter 4]{SerreLF}.
For ramification of order $p^e$, there are $e$ jumps in the filtration which can be indexed in either the 
upper or lower numbering.  The first jump is the same in both numbering systems
and the conductor is one more than the last lower jump.  

\begin{lemma} \label{Lcyclocond}
When $p$ is a regular prime, then the conductor of $L^*/L$ is $p^3-2p^2+2p$.
\end{lemma}

\begin{proof}
The facts in this proof about jumps in ramification filtrations can be found in  \cite[Chapter 4]{SerreLF}.
The extension $K(w)/K$ has jump $p-1$.  
The extension $K(\sqrt[p]{t_i})/K$ has jump $p$ for $1 \leq i \leq r$.
Let $K^{\circ}=K(\sqrt[p]{t_1}, \ldots \sqrt[p]{t_r})$. 
Then $L=K(w) K^{\circ}$.

When $p$ is a regular prime, then there is a unique prime of $L$ above $p$ by Proposition \ref{Prational}.
It is totally ramified since the residue field degrees of $K(w)/K$ and $K(\sqrt[p]{t_i})/K$ are all trivial.
So $L/K$ has $r+1$ upper jumps, and they are $u_1=p-1$ and $u_2=p$ (with multiplicity $r$). 
By Herbrand's formula, the lower jumps $j_1$ and $j_2$ satisfy $j_2-j_1=p(u_2-u_1)$.  
So  $L/K$ has lower jumps $j_1=p-1$ and $j_2=2p-1$ (with multiplicity $r$). 

Thus $L/K(w)$ has lower jump $2p-1$ with multiplicity $r$.
The jump of $K(v)/K(w)$ is $p^2-1$.
Note that $L^* = K(v) K^{\circ}$.  
So $L^*/K(w)$ has upper jumps $U_1=2p-1$ (with multiplicity $r$) and $U_2=p^2-1$. 
By Herbrand's formula, this has lower jumps $J_1=2p-1$ (with multiplicity $r$) and $J_2 =p^3-2p^2+2p-1$.
Thus $L^*/L$ has jump $J_2$ and conductor $J_2+1$.
\end{proof}

We use Lemma \ref{Lcyclocond} in Notation \ref{NgenerateN} and Remark \ref{Ridentify}
to identify the non-trivial homomorphism $\phi^* \in {\rm Ker}(d_2^*)$
in ${\rm Ker}(d_2)$ when $p=3$, in which case the conductor of $L^*/L$ is $15$.

\subsubsection{Kummer extensions}

In this section, we show that some other Kummer extensions of $L$ do not increase the dimension of ${\rm Ker}(d_2)$.

Let $K_0=\mathbb{Q}(\zeta_{p^2})$ and $K_0^* = \mathbb{Q}(\zeta_{p^3})$, 
From Section \ref{Scyclotomic},
$\tau_0$ lifts to an automorphism $\tau_0^*$ of $K_0^*$ such that $\tau_0^*(\zeta_{p^3}) = \zeta_{p^3}^e$ 
for some integer $e$ such that $e \equiv p+1 \bmod p^2$.  Also $\tau_0^*(\zeta_{p^2}) = \zeta_{p^2}^e$. 

Let $F^*_i = K(\sqrt[p^2]{t_i})$ where $t_i = 1 - \zeta_p^{-i}$.
Note that $F^*_i/K$ is not Galois, but $F^*_i K_0$ is Galois over $K_0$.
Also $F^*_i$ is ramified only over $\nu=\langle 1 - \zeta_p \rangle$.

Let $F^*$ be the compositum of $F^*_i$ for $0 \leq i \leq r$.
Let $\tilde{G}^*={\rm Gal}(F^*/K)$.  Then $\tilde{G}^*$ is generated by the lifts $\tau_i^*$ of $\tau_i$, 
each of which has order $p^2$.
Let $G^\circ$ be the subgroup of $\tilde{G}^*$ generated by $\tau_i^*$ for $1 \leq i \leq r$.
Since $K_0^*/K$ is Galois, $G^\circ$ is normal and
there is a short exact sequence
\[1 \to G^\circ \to \tilde{G}^* \to \langle \tau_0^* \rangle \to 1.\]
For $1 \leq i \leq r$, the conjugate of $\tau_i^*$ by $\tau_0^*$ is $(\tau_i^*)^e$ because
\[\tau_0^* \tau_i^* (\tau_0^*)^{-1} (\sqrt[p^2]{t_i}) = \tau_0^*(\zeta_{p^2} \sqrt[p^2]{t_i})=
\zeta_{p^2}^e \sqrt[p^2]{t_i}= (\tau_i^*)^e (\sqrt[p^2]{t_i}).\]

Note that $F^*$ is an elementary abelian $p$-group extension of $L$
which is ramified only above $p$.
So $\bar{N}^*={\rm Gal}(F^*/L)$ is a quotient of $N$.
Let $\bar{a}_i^*$, $\bar{c}_{j,k}^*$ denote the images of $a_i$, $c_{j,k}$ in $\bar{N}^*$.

\begin{proposition} \label{Pkumkerd2}
With notation as above:
\begin{enumerate}
\item $\bar{N}^* \simeq (\mathbb{Z}/p \mathbb{Z})^{r+1}$ and a basis for $\bar{N}^*$ is given by
$\bar{a}_i^*$ for $0 \leq i \leq r$;
\item if $j \not = 0$, then $\bar{c}_{j,k}^*$ is trivial; if $j = 0$, then $\bar{c}_{0,k}^* = \bar{a}_k^*$.
\item If $M$ is a subquotient of the relative homology of the Fermat curve, 
then the kernel of $\bar{d}_2^*$ on $\bar{N}^*$ equals ${\rm Ker}(d_2^*)$ from Proposition \ref{Pcyckerd2}.
\end{enumerate}
\end{proposition}

\begin{proof}
\begin{enumerate}
\item First $\bar{N}^* \simeq \times_{i=0}^r \bar{N_i^*}$ where $\bar{N}_i^* = {\rm Gal}(LK_i^*/L)$.
By Lemma \ref{Lclasscyc}, $\bar{N}_0^* = N^*$ is generated by $a_0^*$.
For $1 \leq i \leq r$, the image of $\bar{a}_i^*$ in $\bar{N}_i^*$ is non-trivial because 
$(\tau_i^*)^p (\sqrt[p^2]{t_i})/\sqrt[p^2]{t_i} = \zeta_p$;
but the image of $\bar{a}_i^*$ in $\bar{N}_j^*$ for $i \not = j$ is trivial.

\item If $j \not = 0$, then $\bar{c}_{j,k}^*$ is trivial because $\tau_j^*$ and $\tau_k^*$ commute.
If $j=0$, then 
\[\tau_0^* \tau_k^* (\tau_0^*)^{-1} (\tau_k^*)^{-1}(\sqrt[p^2]{t_i})/\sqrt[p^2]{t_i}= (\tau_k^*)^{e-1}(\sqrt[p^2]{t_i})/(\sqrt[p^2]{t_i})
= \zeta_{p^2}^{e-1} = \zeta_p.\]

\item  Suppose $\phi \in \bar{d}_2^*$.  Then $\phi(\bar{a}_i^*) = 0$ for $1 \leq i \leq r$.
So $\phi$ is zero on $\times_{i=1}^r \bar{N}^*$, and is thus determined by its image on $\bar{N}_0^*$
\end{enumerate}
\end{proof}

\subsection{Review of Heisenberg extensions} 

Let $H_p$ denote the mod $p$ Heisenberg group, namely the multiplicative group of upper triangular $3 \times 3$ matrices with coefficients in $\mathbb{Z}/p \mathbb{Z}$ and diagonal entries equal to $1$.  
Let $U_p$ denote the central subgroup of $H_p$ consisting of those matrices for which the upper right corner is the only non-zero entry off the diagonal.

Let $q: H_p \to H_p/U_p \simeq \mathbb{Z}/p \mathbb{Z} \times \mathbb{Z}/p \mathbb{Z}$ denote the quotient map.
The two coordinate projections $\mathbb{Z}/p \mathbb{Z} \times \mathbb{Z}/p \mathbb{Z} \to \mathbb{Z}/p \mathbb{Z}$ produce two classes 
$\iota_1$, $\iota_2$ in $\rm{H}^1( \mathbb{Z}/p \mathbb{Z} \times \mathbb{Z}/p \mathbb{Z}, \mathbb{Z}/p \mathbb{Z})$. The 
cup product $\iota_1 \cup \iota_2$ in $\rm{H}^2( \mathbb{Z}/p \mathbb{Z} \times \mathbb{Z}/p \mathbb{Z}, \mathbb{Z}/p \mathbb{Z})$ classifies 
the extension \begin{equation}\label{Hextension}1\to U_p \to H_p \stackrel{q}{\to} \mathbb{Z}/p \mathbb{Z} \times \mathbb{Z}/p \mathbb{Z} \to 1.\end{equation}

Heisenberg extensions appear in many places in the literature.  We follow the notation of 
\cite{Sharifi}.
The next result is a special case of \cite[Proposition 2.3]{Sharifi}.  Let $K=\mathbb{Q}(\zeta_p)$.

\begin{proposition}\label{makingHexts} 
Suppose $L_{\alpha,\beta} = K(\sqrt[p]{\alpha}, \sqrt[p]{\beta})$ is a field extension of $K$ with 
${\rm Gal}(L_{\alpha, \beta}/K) \simeq \mathbb{Z}/p \mathbb{Z}\times \mathbb{Z}/p \mathbb{Z}$.
Then there is a Galois field extension $R/K$ dominating $L_{\alpha, \beta}/K$ 
such that $\operatorname{Gal}(R/K) \to \operatorname{Gal}(L_{\alpha, \beta}/K)$ is isomorphic to 
$q: H_p \to \mathbb{Z}/p \mathbb{Z} \times \mathbb{Z}/p \mathbb{Z}$ if and only if $\kappa(\alpha) \cup \kappa(\beta) = 0$ in 
$\rm{H}^2(G_K, (\mathbb{Z}/p \mathbb{Z})(2)) \cong \rm{H}^2(G_K, \mathbb{Z}/p \mathbb{Z})$.
\end{proposition}

Furthermore, by \cite[Section 2.4]{Sharifi}, the extension $R/K$ can be constructed explictly, and in some sense uniquely, as follows.
Let $K_\alpha=K(\sqrt[p]{\alpha})$ and $K_\beta=K(\sqrt[p]{\beta})$.  Then $L_{\alpha,\beta}=K_\alpha K_\beta$.
Let $\tau_\alpha \in \operatorname{Gal}(K_\alpha/K)$ be multiplication by $\zeta_p$ on $\sqrt[p]{\alpha}$ and 
let $\tau_\beta \in \operatorname{Gal}(K_\beta/K)$ be multiplication by $\zeta_p$ on $\sqrt[p]{\beta}$.
These determine 2 characters $\chi_\alpha$ and $\chi_\beta$.

Consider the surjection $\bar{\rho}:G_K \to \operatorname{Gal}(L_{\alpha, \beta}/K)$.
By \cite[Section 2.4]{Sharifi}, $\bar{\rho}$ lifts to $\rho:G_K \to H_p$ if and only if the cup product 
$\chi_\alpha \cup \chi_\beta$ is zero.  Also, the cup product equals the norm residue symbol $(\alpha,\beta)$
which is trivial if and only if $\beta \in N_{K_\alpha/K}(K_\alpha^*)$ 
\cite[XIV.2]{SerreLF}.  
In this case, let $\underline{\beta} \in K_\alpha^*$ be such that $\beta= N_{K_\alpha/K}(\underline{\beta})$.
Then, let
\begin{equation} \label{Egammagen}
\gamma_{\alpha,\beta}=\prod_{j=0}^{p-1} \tau_\alpha^j(\underline{\beta})^j.
\end{equation}
By \cite[Lemma 2.4]{Sharifi}, 
$\sigma(\gamma_{\alpha, \beta}) \equiv \gamma_{\alpha, \beta} \bmod (L_{\alpha, \beta}^*)^p$ for all $\sigma \in {\rm Gal}(L_{\alpha,\beta}/K)$.

By \cite[Theorem 2.5]{Sharifi}, the Heisenberg relation $\rho$ has the property that, for all $\xi \in G_{L_{\alpha,\beta}}$,
\[\rho(\xi)={\rm Ind} \frac{\xi(\sqrt[p]{\gamma_{\alpha, \beta}})}{\sqrt[p]{\gamma_{\alpha, \beta}}}.\]
This means that $\rho$ factors through the extension $R_{\alpha,\beta}=L_{\alpha,\beta}(\sqrt[p]{\gamma_{\alpha,\beta}})$.
Furthermore, $\gamma_{\alpha, \beta}$ is unique up to multiplication by an element of $K^* (L_{\alpha, \beta}^*)^p$. 

Finally, by \cite[Equation 2.4]{Sharifi}, 
\begin{equation} \label{Eactiongen}
\tau_\alpha(\gamma_{\alpha, \beta})=\frac{\underline{\beta}^p}{N_{K_\alpha/K}(\underline{\beta})}\gamma_{\alpha, \beta}.
\end{equation}

\subsection{Heisenberg extensions of $K$} \label{Sheisenberg}

We apply this to the $(\mathbb{Z}/p \mathbb{Z})^2$-extensions of $K$ in $L$.  
The Steinberg relation is that the cup product $\kappa(\alpha) \cup \kappa(1-\alpha) = 0$ is zero
for any $\alpha \in K^* -\{1\}$, see \cite[section 11]{Milnor}.
Let $1 \leq I \leq r$.
Choose $\alpha=\zeta_p^{-I}$ and $\beta=1-\zeta_p^{-I}$. Let
\[F_I=K(\sqrt[p]{\zeta_p^{-I}}, \sqrt[p]{1-\zeta_p^{-I}}).\]

Applying Proposition \ref{makingHexts}, there is a field extension $R_{I}/K$ dominating 
$F_I/K$ such that $\operatorname{Gal}(R_I/K) \to \operatorname{Gal}(F_I/K)$ is isomorphic to $q: H_p \to \mathbb{Z}/p \mathbb{Z} \times \mathbb{Z}/p \mathbb{Z}$.

\begin{lemma} \label{Lbetanorm}
Let $w=\zeta_{p^2}$ and $\underline{\beta}_I=1-w^{-I}$.  Let $K_\alpha=K(\sqrt[p]{\zeta_p^{-I}})=K(w)$.  
Then $N_{K_\alpha/K}(\underline{\beta}_I) = 1 - \zeta_p^{-I}$.
\end{lemma}

\begin{proof}
By definition, $N_{K_\alpha/K}(\underline{\beta}_I)=\prod_{\ell=0}^{p-1} \tau_0^\ell (\underline{\beta}_I)$.
Because $\tau_0(w)=\zeta_p w$,
this simplifies to 
\begin{equation} \label{Ebetaidentity}
(1-w^{-I})(1-\zeta_p^{-I} w^{-I}) \cdots (1-\zeta_p^{-I(p-1)} w^{-I})=1-\zeta_p^{-I}.
\end{equation}
\end{proof}

By definition, $\tau_\beta$ acts by multiplication by $\zeta_p$ on $\sqrt[p]{\beta}$ where 
$\beta=1-\zeta_p^{-I}$.  So $\tau_\beta=\tau_{I}$.
By definition, $\tau_\alpha$ acts by multiplication by $\zeta_p$ on $\sqrt[p]{\alpha}$ where 
$\alpha=\zeta_p^{-I}$.  So $\tau_\alpha=\tau_0^J$ where $J$ is such that $-IJ \equiv 1 \bmod p$.

In particular, $\tau_\alpha(w) = \zeta_p^J w$.
Write 
$\underline{\beta}=\underline{\beta}_I=1-w^{-I}$.
So 
\[\tau_\alpha^j(\underline{\beta}) = 1 - (\zeta_p^{jJ}w)^{-I} = 1 - \zeta_p^j w^{-I}.\] 
 
Let $\gamma_I = \gamma_{\alpha, \beta}$.  
By \eqref{Egammagen},
\begin{equation} \label{Egamma}
\gamma_I=\prod_{j=0}^{p-1} \tau_\alpha^j(\underline{\beta})^j
=\prod_{j=1}^{p-1} (1 - \zeta_p^{j} w^{-I})^j=\prod_{j=1}^{p-1} (1-w^{pj-I})^j.
\end{equation}

By \eqref{Eactiongen} and Lemma \ref{Lbetanorm},
\begin{equation} \label{Egaloisactionc}
\tau_\alpha(\gamma_I)=\frac{(1-w^{-I})^p}{1-\zeta_p^{-I}}\gamma_I.
\end{equation}

Let $\tilde{R}_I = L R_I$.
Both $R_I/F_I$ and $\tilde{R}_I/L$ are generated by $\sqrt[p]{\gamma_I}$. 

\begin{lemma} \label{LdominateH}
The Heisenberg extension $R_{I}/K$ is ramified only over $1 - \zeta_p$.  
Thus $\tilde{R}_{I} \subset E$.
\end{lemma}

\begin{proof}
The field $\tilde{R}_I$ is a $\mathbb{Z}/p \mathbb{Z}$-Galois extension of $L$.
The second statement follows directly from the first, since it guarantees that 
the ramification of $\tilde{R}_I$ occurs only at primes above $p$.

Recall that $F_I$ over $K$ is ramified only over $p$ \cite[Lemma 3.2]{WINF:birs}.
So it suffices to prove that all the ramification of $R_{I}$ over $F_I$ lies above $p$.
By \cite[Lemma~5, Section 3.2]{CasselsFrohlich}, a prime $\eta$ of $F_I$ is unramified in $R_{I}$ if $\eta \nmid p \gamma_I$.
But $\gamma_I$ is a product of powers of the conjugates of $\underline{\beta}_I$ under $\sigma$.
Each of these conjugates of $\underline{\beta}_I$ is a generator for the unique prime ideal of $K_I$ above $p$.
Thus the primes of $F_I$ which are ramified in $R_{I}$ all lie above $\underline{\beta}_I$ which lies above $p$.  Since there are no real places of $K$, there is no ramification of $K$ over any infinite place.
\end{proof}

The following result is useful because the conductor is the same as the index of the modulus
for which this extension appears in the ray class field. 

\begin{lemma} \label{Lheiscond}
The conductor of $R_I/F_I$ is $p^2 + p(p-1)/2$.
\end{lemma}

\begin{proof}
Take $z=\sqrt[p]{\gamma_I}$.  Then $z$ generates $R_I/F_I$.
The conductor is the valuation in $F_I$ of $g(z)-z = (\zeta_p-1)z$, where $g$ generates ${\rm Gal}(R_I/F_I)$. 
The valuation of $z$ in $F_I$ equals the valuation of 
$\gamma_I$ in $\mathbb{Z}[\zeta_{p^2}]$ which is $1 + \cdots + (p-1)=p(p-1)/2$.
So the conductor is $p^2 + p(p-1)/2$.  
\end{proof}

We use Lemma \ref{Lheiscond} in Notation \ref{NgenerateN} and Remark \ref{Ridentify} when $p=3$, 
with conductor $12$. 

\subsection{The Heisenberg classifying element} \label{Sheisclass}

Recall that $R_{I}/K$ is an $H_p$-Galois extension dominating $F_I/K$ and 
$\tilde{R}_I = L R_I$.
Let $N_I = {\rm Gal}(\tilde{R}_I/L) = {\rm Gal}(R_I/F_I) \simeq \mathbb{Z}/p \mathbb{Z}$.

Let $\tilde{R}$ be the compositum of $\tilde{R}_I$ for $I \in {\mathcal I}:=\{1, \ldots, (p-1)/2\}$.  
Note that $\tilde{R}/K$ is Galois because the 
action of $Q$ permutes the fields $F_I$, and thus permutes the Heisenberg extensions $R_I$.  
Thus $Q$ stabilizes $\tilde{R}$.
Let $\bar{N}_H = {\rm Gal}(\tilde{R}/L)$.  

\begin{proposition} \label{Ptrivialinquotient}
Let $\bar{a}_i, \bar{c}_{j,k}$ be the images of $a_i, c_{j,k}$ in $\bar{N}_H$.
Then $\bar{N}_H \simeq \times_{I \in {\mathcal I}} N_I \simeq (\mathbb{Z}/p\mathbb{Z})^r$
and a basis for $\bar{N}_H$ is given by $\{\bar{c}_{0,k} \mid 1 \leq k \leq r\}$.  In particular, 
\begin{enumerate}
\item $\bar{a}_i$ is trivial for $0 \leq i \leq r$;
\item $\bar{c}_{j,k}$ is trivial if $j \not = 0$;
\item and the image of 
$\bar{c}_{0,k}$ in $N_I$ is non-trivial if and only if $k=I$.  
\end{enumerate}
\end{proposition}

\begin{proof}
With some risk of confusion, we use the same notation $\bar{a}_i$ and $\bar{c}_{j,k}$ 
to denote the images of $\bar{a}_i$ and $\bar{c}_{j,k}$ in ${\rm Gal}(R_I/F_I)$.
The claim is that $\bar{a}_i$ and $\bar{c}_{j,k}$ are trivial for each $I$ when $j \not = 0$ 
and that $\bar{c}_{0,k}$ is non-zero if and only if $k=I$.

The last part of this claim implies the main statement of Proposition \ref{Ptrivialinquotient}. 
If $\bar{c}_{0,I}$ has a non-trivial image in $N_I$ but a trivial image in $N_k$ for $k \not = I$,
then $R_{I}$ is disjoint from the compositum of 
$\{R_{k} \mid 1 \leq k \leq r, k \not = I\}$.

Let $w=\zeta_{p^2}$ and let $t_{I}=\beta$.  
Recall that $F_I=K(w, \sqrt[p]{t_{I}})$ and that ${\rm Gal}(F_I/K)$ is generated by $\sigma=\tau_0$ and 
$\tau_\beta=\tau_{I}$ where
$\sigma(w)=\zeta_p w = w^{1+p}$
and $\sigma(\sqrt[p]{t_{I}})=\sqrt[p]{t_{I}}$ and $\tau_\beta(w)=w$ and 
$\tau_\beta(\sqrt[p]{t_{I}})=\zeta_p \sqrt[p]{t_{I}}$.
For $\ell \not = 0,I$, the other automorphisms $\tau_\ell$ generating $Q$ act trivially on $F_I$. 

Recall that $R_I/K$ is a Heisenberg extension and that $\sigma$ and $\tau_1, \ldots, \tau_r$ 
extend to $\tilde{\sigma}$ and $\tilde{\tau}, \ldots, \tilde{\tau}_r$ in ${\rm Gal}(R_I/K)$.
By definition, the elements $\bar{a}_i$ and $\bar{c}_{j,k}$ in $N_I$ are
\[\bar{a}_i=\tilde{\tau}_i^p(\sqrt[p]{\gamma_I})/\sqrt[p]{\gamma_I},\]
and
\[\bar{c}_{j,k}=\tilde{\tau}_k\tilde{\tau}_j \tilde{\tau}_k^{-1} \tilde{\tau}_j^{-1}(\sqrt[p]{\gamma_I})/\sqrt[p]{\gamma_I}.\]

\begin{enumerate}

\item Then $\bar{a}_i$ is trivial because $\tilde{\tau}_i$ has order $p$ in the Heisenberg group $H_p$.

\item If $j \not = 0$, then $\tilde{\tau}_j$ and $\tilde{\tau}_k$ fix $\sqrt[p]{\gamma_I}$, 
so $\bar{c}_{j,k}$ is trivial.

\item 
We compute $\bar{c}_{0,k}$ in $N_I$. 
Now 
$\bar{c}_{0,k} = [\tilde{\tau}_k, \tilde{\sigma}] (\sqrt[p]{\gamma_I})/\sqrt[p]{\gamma_I}$.
Since 
$\tilde{\sigma} =\tilde{\tau}_\alpha^{-I}$, the following quantity is non-trivial exactly when
$\bar{c}_{0,k}$ is non-trivial:
\begin{equation} \label{Etrivialc}
[\tilde{\tau}_k, \tilde{\tau}_\alpha] (\sqrt[p]{\gamma_I})/\sqrt[p]{\gamma_I}
=
\tilde{\tau}_k\tilde{\tau}_\alpha \tilde{\tau}_k^{-1} \tilde{\tau}_\alpha^{-1}(\sqrt[p]{\gamma_I})/\sqrt[p]{\gamma_I}.
\end{equation}

By \eqref{Egaloisactionc}, for some $z \in \mu_p$,
\[\tilde{\tau}_\alpha(\sqrt[p]{\gamma_I}) = z \frac{1-w^{-I}}{\sqrt[p]{1-\zeta_p^{-I}}} \sqrt[p]{\gamma_I} = z \frac{1-w^{-I}}{\sqrt[p]{t_I}} \sqrt[p]{\gamma_I} .\]
The value of $z$ is not important since it is fixed by $Q$; 
set $z=1$. 

If $k \not = I$,
then $\tilde{\tau}_k$ fixes $\sqrt[p]{t_I}$ and $\omega$ and thus commutes with the action of 
$\tilde{\tau}_\alpha$ on $\sqrt[p]{\gamma_I}$. 
Thus the quantity in \eqref{Etrivialc} is trivial and $\bar{c}_{0,k}$ is trivial in $N_I$ when $k \not = I$.

Now consider the case that $k=I$.
The result is stated in \cite[Equation~2.5]{Sharifi}; 
since no details are included there, 
we include a proof for the benefit of the reader. 

Note that $\tilde{\tau}_\alpha$ fixes $\sqrt[p]{t_k}$.  Thus
\[\tilde{\tau}_\alpha^2(\sqrt[p]{\gamma_k}) = 
\frac{(1- \tilde{\tau}_\alpha(w)^{-k})}{\sqrt[p]{t_k}} \frac{(1-w^{-k})}{\sqrt[p]{t_k}} \sqrt[p]{\gamma_k}.\]
It follows that
\[\tilde{\tau}_\alpha^{p-1}(\sqrt[p]{\gamma_k}) = 
\frac{1}{(\sqrt[p]{t_k})^{p-1}}(1-w^k)(1-\tilde{\tau}_\alpha(w^k)) \cdots (1-\tilde{\tau}_\alpha^{p-2}(w^{k}))\sqrt[p]{\gamma_k}.\]

Recall that $\tilde{\tau}_k$ acts by multiplication by $\zeta_p$ on $\sqrt[p]{t_k}$. 
By modifying $\tilde{\tau}_k$ by an element of ${\rm Gal}(\tilde{R}/L)$, 
we may assume that $\tilde{\tau}_k(\sqrt[p]{\gamma_k})=\zeta_p\sqrt[p]{\gamma_k}$.
It follows that
\begin{equation} \label{Elonglong}
\tilde{\tau}_k^{p-1} \tilde{\tau}_\alpha^{p-1}(\sqrt[p]{\gamma_k}) = 
\frac{1}{\zeta_p (\sqrt[p]{t_k})^{p-1}}(1-w^k)(1-\tilde{\tau}_\alpha(w^k)) \cdots (1-\tilde{\tau}_\alpha^{p-2}(w^{k}))
\zeta_p \sqrt[p]{\gamma_k}.
\end{equation}
Applying $\tilde{\tau}_\alpha$ to the right hand side of \eqref{Elonglong} yields
\begin{equation}
\frac{1}{(\sqrt[p]{t_k})^{p-1}}(1-\tilde{\tau}_\alpha(w^k))(1-\tilde{\tau}^2_\alpha(w^k)) \cdots (1-\tilde{\tau}_\alpha^{p-1}(w^{k})) \frac{1-w^{-k}}{t_k} \sqrt[p]{\gamma_k}.
\end{equation}
By Lemma \ref{Lbetanorm},
\[\tilde{\tau}_\alpha \tilde{\tau}_k^{p-1} \tilde{\tau}_\alpha^{p-1}(\sqrt[p]{\gamma_k}) = \frac{1}{t_k} N_{K_\alpha/K}(1-w^k) \sqrt[p]{\gamma_k} =  \sqrt[p]{\gamma_k}.\]
Thus 
\[[\tilde{\tau}_k, \tilde{\tau}_\alpha] (\sqrt[p]{\gamma_k}) = 
\tilde{\tau}_k \tilde{\tau}_\alpha \tilde{\tau}_k^{p-1} \tilde{\tau}_\alpha^{p-1}(\sqrt[p]{\gamma_k}) 
= \zeta_p \sqrt[p]{\gamma_k}.\]
Thus $\bar{c}_{0,k}$ is non-trivial in $N_k$, which completes the proof.
\end{enumerate}
\end{proof}

\begin{proposition} \label{PsimplifyH1}
If $M$ is a $G$-module on which $\bar{N}_H$ acts trivially, 
then $\rm{H}^1(\bar{N}_H, M)^Q \simeq (M^Q)^r$.
\end{proposition}

\begin{proof}
Since $M$ is a trivial $\bar{N}_H$-module,
$\rm{H}^1(\bar{N}_H, M)^Q \simeq \mathrm{Hom}(\bar{N}_H, M)^Q$.
By Proposition  \ref{Ptrivialinquotient},
$\bar{N}_H$ has basis $\{\bar{c}_{0,k} \mid 1 \leq k \leq r\}$.
Thus $\bar{\phi}$ is determined uniquely by the values 
$\mu_k=\bar{\phi}(\bar{c}_{0,k}) \in M$.
Since $\bar{N}_I \simeq U_p$ is central in $H_p$, 
the homomorphism $\bar{\phi}$ is $Q$-invariant if and only if 
$\mu_k \in M^Q$ for $1 \leq k \leq r$ by Lemma \ref{Lcentral}.
\end{proof}

\subsection{The kernel of $d_2$ for Heisenberg extensions} \label{SmainHes}

Consider the map \[d_{2, H}: \rm{H}^1(\bar{N}_H, M)^Q \to \rm{H}^2(Q, M).\]
By Lemma \ref{Ld2d2bar}, ${\rm Ker}(d_{2, H}) \subset {\rm Ker}(d_2)$.
By Proposition \ref{PsimplifyH1}, we can study the isomorphic image
of ${\rm Ker}(d_{2,H})$ in $(M^Q)^r$. 

\begin{theorem} \label{PHeisKer}
Let $X$ be the Fermat curve of degree $p$.  
Let $M$ be a subquotient of $\rm{H}_1(U, Y; \mathbb{Z}/p \mathbb{Z})$.
Then 
${\rm Ker}(d_{2,H})$ is isomorphic to the set of all $(\mu_1, \ldots, \mu_r) \in (M^Q)^r$ such that 
\[(\mu_1, \ldots, \mu_r) = ((1-\tau_1)m_0 - (1-\tau_0)m_1, \ldots, (1-\tau_r)m_0 - (1-\tau_0)m_r),\]
for some $m_0, \ldots, m_r \in M$ such that $(1-\tau_k)m_j - (1-\tau_j)m_k = 0$ for all $1 \leq j < k \leq r$.
When $p=3$, we further require that $y_0^2y_1^2 m_0 = 0$.
\end{theorem}

\begin{proof}
By Proposition \ref{PsimplifyH1}, there is an isomorphism $\rm{H}^1(\bar{N}_H, M)^Q \to (M^Q)^r$, 
where $\bar{\phi} \mapsto (\mu_1, \ldots, \mu_r)$ where 
$\mu_k =\bar{\phi}(\bar{c}_{0,k})$.
Recall the conditions on the tuple $(\mu_1, \ldots, \mu_r)$ in Theorem \ref{Tkerd2} 
which are equivalent to $\bar{\phi}$ being in ${\rm Ker}(\bar{d}_{2,H})$.
By Theorem \ref{Tnorm}, the constraint $\bar{\phi}(\bar{a}_i) = - N_{\tau_i} m_i$ gives no constraint when $p \geq 5$, because both sides equal zero;
when $p=3$, then the constraint is only satisfied if $0 =\bar{\phi}(\bar{a}_0)  =- y_0^2y_1^2 m_0$.
\end{proof} 

\begin{remark}
By \cite[9.6 and 10.5.2]{Anderson}, $1-\tau_k \in \langle y_0y_1\rangle$.
By \cite[Proposition~6.2]{WINF:birs}, $\langle y_0y_1\rangle \simeq \rm{H}_1(U; \mathbb{Z}/p \mathbb{Z})$.
So $(\mu_1, \ldots, \mu_r) \in (\rm{H}_1(U; \mathbb{Z}/p \mathbb{Z})^Q)^r$.
\end{remark}

\subsection{Computing ${\rm Ker}(d_{2, H})$ for small $p$} \label{Sheisenbergbig} 

We compute ${\rm Ker}(d_{2, H})$ for small $p$ by finding matrices for the action of 
${\rm Gal}(L/K)$ on $\rm{H}_1(U,Y; \mathbb{Z}/p \mathbb{Z})$; the explicit formulas for this action are found in \cite[Theorem 3.5]{WINF:Baction}.
From this, we determine matrices for the action on $\rm{H}_1(U; \mathbb{Z}/p \mathbb{Z})$ and $\rm{H}_1(X; \mathbb{Z}/p \mathbb{Z})$.
We did these computations using Magma \cite{Magma}.

\begin{example} \label{Eheiskerd23}
When $p=3$, then $\bar{\phi} \in \rm{H}^1(\bar{N}_H, M)$ is determined by 
$\mu_1=\bar{\phi} (\bar{c}_{0,1})$.
Also, $\bar{\phi}$ is $Q$-invariant if and only if $\mu_1 \in M^Q$.

Now $\bar{\phi} \in {\rm Ker}(d_{2,H})$
if and only if $\mu_1$ is in the image of the map 
$T: M^2 \rightarrow M$ given by $(m_0, m_1) \mapsto (1-\tau)m_0 - (1-\sigma)m_1$  (condition $c_{0,1}$)
and $y_0^2y_1^2 m_0 = 0$ (condition $a_0$).  

For $p=3$, we compute the dimension of ${\rm Ker}(d_{2, H})$ for several choices of $M$: 
\[\begin{array}{ | c | c | c | c | } 
\hline
 M  & \rm{H}_1(U, Y) & \rm{H}_1(U) & \rm{H}_1(X) \\ \hline
  \mathrm{dim}(M^Q) & 5& 3& 2 \\ \hline
  \mathrm{dim}(\mathrm{im}(T \mid_{{\rm cond} \ a_0})) & 4& 1& 0 \\ \hline
 \mathrm{dim} (\mathrm{Ker}(d_{2,H}))= \mathrm{dim} (M^Q \cap \mathrm{im}(T \mid_{{\rm cond} \ a_0})) & 3 & 1 &0 \\ \hline
\end{array}.\]

In particular, when $M=H_1(U,Y; \mathbb{Z}/3\mathbb{Z})$, then $\bar{\phi} \in {\rm Ker}(\bar{d}_{2,H})$ if and only if 
$\mu_1 \in \rm{H}_1(U)^Q$.  
\end{example}

\begin{example}
When $p=5$, then $\bar{\phi} \in \rm{H}^1(\bar{N}_H, M)$ is determined by 
$\mu_1=\bar{\phi} (\bar{c}_{0,1})$ and $\mu_2=\bar{\phi} (\bar{c}_{0,2})$.
Also, $\bar{\phi}$ is $Q$-invariant if and only if $\mu_1, \mu_2 \in M^Q$.

Now $\bar{\phi} \in {\rm Ker}(d_{2,H})$
if and only if $(\mu_1, \mu_2)$ is in the image of the map 
$T: M^3 \rightarrow M^2$ given by
\[(m_0, m_1, m_2) \mapsto ((1-\tau_1)m_0 - (1-\tau_0)m_1, (1-\tau_2)m_0 - (1-\tau_0)m_2),\] 
(condition $c_{0,1}$, $c_{0,2}$) and
$(1-\tau_2)m_1-(1-\tau_1)m_2=0$ (condition $c_{1,2}$).
For $p=5$, we compute:
\[\begin{array}{ | c | c | c | c | } 
\hline
 M  & \rm{H}_1(U; Y) & \rm{H}_1(U) & \rm{H}_1(X) \\ \hline
  \mathrm{dim}(M^Q) & 11 & 9 & 8 \\ \hline
  \mathrm{dim}((M^Q)^r) & 22 & 18 & 16 \\ \hline
  \mathrm{dim}(\mathrm{im}(T \mid_{{\rm cond} \  c_{1,2}})) & 16  &  8 & 4 \\ \hline
 \mathrm{dim}(\mathrm{Ker}(d_{2,H}))= \mathrm{dim}((M^Q)^r \cap \mathrm{im}(T \mid_{{\rm cond} \  c_{1,2}})) & 11  & 7 & 4  \\ \hline
\end{array}.\]
\end{example}

\begin{example}
For $p=7$, we compute:
\[\begin{array}{ | c | c | c | c | } 
\hline
 M  & \rm{H}_1(U; Y) & \rm{H}_1(U) & \rm{H}_1(X) \\ \hline
  \mathrm{dim}(M^Q) & 17  & 15 & 14  \\ \hline
  \mathrm{dim}((M^Q)^r) & 51 & 45 & 42 \\ \hline
  \mathrm{dim}(\mathrm{im}(T \mid_J)) & 36  & 23  & 16 \\ \hline
 \mathrm{dim} (\mathrm{Ker}(d_{2,H}))= \mathrm{dim}((M^Q)^r \cap \mathrm{im}(T \mid_J)) & 19  & 14 & 10  \\ \hline
\end{array}.\]
\end{example}

Here the image of $T:M^4 \to M^3$ are the triples $(\mu_1, \mu_2, \mu_3)$ satisfying conditions 
$c_{0,1}$, $c_{0,2}$, and $c_{0,3}$, and $J$ denotes the conditions coming from $c_{1,2}$, $c_{1,3}$, and $c_{2,3}$.

\begin{remark} In Theorem \ref{PHeisKer}, taking $m_0 \in M$ and $m_i = 0$ for $1 \leq i \leq r$, 
we see that ${\rm Ker}(d_{2,H})$ contains a subspace isomorphic to the set
\[\{(\mu_1, \ldots, \mu_r) \in (M^Q)^r \mid  (\mu_1, \ldots, \mu_r) = ((1-\tau_1)m_0, \ldots, (1-\tau_r)m_0)\}.\]
Thus, 
\begin{equation} \label{Elowerbound}
{\rm dim}({\rm Ker}(d_{2,H})) \geq \mathrm{dim} (\mathrm{Im}((B_{\tau_1}-1) \displaystyle |_M) \cap M^Q).
\end{equation} 
When $p=3,5,7$ and when $M=\rm{H}_1(X)$, the lower bound in \eqref{Elowerbound} is in fact an equality. 
\end{remark}




\subsection{Other unitary extensions}

Recall that $R_I/K$ is a Heisenberg degree $p^3$ extension for $1 \leq I \leq r$.
Consider the group $U_4$ which has order $p^6$ and exponent $p$.
Using the results of \cite[Section 3.1]{minactan}, 
it is possible to construct a $U_4$-Galois extension of $K$ from 
$R_I/K$ and $R_J/K$, when $1 \leq I < J \leq r$.  
For $p \geq 5$, this yields a degree $p^3$ elementary abelian $p$-group extension $\bar{E}_U/L$.  
Let $\bar{N}_U = {\rm Gal}(\bar{E}_U/L)$.
By \cite[Claim, page 1036-1037]{minactan}, the lifts of $\tau_I$ and $\tau_J$ commute in $U_4$.
This implies that the image of $c_{I,J}$ is trivial in $\bar{N}_U$.
From this, we expect that ${\rm dim}({\rm Ker}(d_2))$ grows by at least ${\rm dim}(M^Q)$
in the passage from $\bar{N}_H$ to $\bar{N}_U$.

\section{The kernel of the transgression map when $p=3$} \label{Supperbound}

In this section, $p=3$.
Then $K=\mathbb{Q}(\zeta_3)$ and $L$ is the splitting field of $1-(1-x^3)^3$.
Let $\sigma=\tau_0$ and $\tau=\tau_1$ be the generators of 
$Q ={\rm Gal}(L/K) \simeq (\mathbb{Z}/3 \mathbb{Z})^2$ from Section \ref{Sgalois}.

Then $E/L$ is the maximal elementary abelian $3$-group extension of $L$ ramified only over $3$.
Also $G={\rm Gal}(E/K)$ and $N={\rm Gal}(E/L)$.
In Corollary \ref{PrankN3}, we show that ${\rm dim}_{\mathbb{F}_3}(N)=10$.

Let $M=\rm{H}_1(U, Y; \mathbb{Z}/3 \mathbb{Z})$ be the relative homology of the Fermat curve of degree $3$. 
In Lemma \ref{LbasisA}, we find a basis for $\rm{H}^1(N, M)^Q$; it has dimension $18$.

In Proposition \ref{Pclass3}, we determine the element $\omega \in \rm{H}^2(Q, N)$ classifying the exact sequence 
\begin{equation} \label{Eshortexactsec4}
1 \to N \to G \to Q \to 1.
\end{equation}
Recall that $d_2: \rm{H}^1(N, M)^Q \to \rm{H}^2(Q, N)$ is the transgression map.

The main result of the section is Corollary \ref{Cfinish3}, in which we determine ${\rm Ker}(d_2)$ completely 
when $p=3$: in particular, we show
that it has dimension $5$ and is determined by a degree $3^5$  
extension of $K$ whose Galois group is non-abelian and has exponent $9$; replacing $M$ by $\rm{H}_1(X; \mathbb{Z}/3\mathbb{Z})$, 
we show that ${\rm Ker}(d_2)$ is determined by the Heisenberg extension of $K$.

Similar calculations for $p=5$ appear out of reach, since ${\rm deg}(L/{\mathbb Q})=500$ in that case.

\subsection{Ray class fields when $p=3$} \label{Srayclass3}

A Magma computation shows that $\mathop{\rm Cl}(L)$ is trivial.
By Proposition \ref{Prational}, there is a unique prime $\mathfrak{p}$ of $L$ above $p$.
We compute that $\mathfrak{p} = \langle \zeta_{9}, \sqrt[3]{t_1} \rangle$.

Let $L_{\mathbf{m}}$ (resp.\ $\mathop{\rm Cl}_{\mathbf{m}}(L)$) denote the ray class field (resp.\ group) of $L$ of modulus $\mathbf{m}$.

\begin{corollary} \label{PrankN3}
When $p=3$, then $N=\mathop{\rm Cl}_{\mathfrak{p}^{28}}(L)$ and ${\rm dim}_{\mathbb{F}_3}(N)=10$.
\end{corollary}

\begin{proof}
When $p=3$, then $e=d=18$, $f=1$, and $e_1 = 9$.
By Proposition \ref{Prational}, ${\rm dim}_{\mathbb{F}_3}(N)=10$.  
By Lemma \ref{rrbound}, $(\mathcal{O}_L/\mathfrak{p}^i)^{\times}$ has 3-rank 19 for $i \geq 28$.
A Magma computation \cite{Magma} shows that $\mathop{\rm Cl}_{\mathfrak{p}^{28}}(L) \simeq (\mathbb{Z}/3\mathbb{Z})^{10}$;
in particular, $\mathop{\rm Cl}_{\mathfrak{p}^{28}}(L)$ has $3$-rank $10$.
By Lemmas \ref{rrbound} and \ref{Lstabilize}, since the rank does not increase for modulus 
$\mathfrak{p}^i$ for $i > 28$, $N = \mathop{\rm Cl}_{\mathfrak{p}^{28}}(L)$. 
\end{proof}

More generally, we compute the $3$-rank of the 
ray class group $\mathop{\rm Cl}_{\mathfrak{p}^i}(L)$ with modulus $\mathfrak{p}^i$ for $1 \leq i \leq 28$. 
The rank increases at modulus $\mathfrak{p}^{i}$ when $i$ is one of the following values $m_1, \dots, m_{10}$:
$12, 15, 18, 20, 21, 23, 24, 26, 27, 28$.

\begin{notation} \label{NgenerateN}
For $1 \leq \ell \leq 10$, let $L_\ell$ denote the maximal elementary abelian extension of $L$ of modulus $\mathfrak{p}^{m_\ell}$ 
and consider $N_\ell={\rm Gal}(L_\ell/L)$. 
By Lemma~\ref{Lheiscond}, $N_1=\bar{N}_H$, where $\bar{N}_H$ is
the quotient of $N$ arising from the Heisenberg extension of $K$.
By Lemma~\ref{Lcyclocond}, $N_2=N_1N^*$, where $N^*$ is 
the quotient of $N$ arising from the cyclotomic extension of $K$.
\end{notation}

\subsection{Computation of $\rm{H}^1(N,M)^Q$} \label{SNMQ}

Let $M=\rm{H}_1(U,Y; \mathbb{Z}/3 \mathbb{Z})$.  Then ${\rm dim}_{\mathbb{F}_3}(M)=9$.  
Recall the definition of $y_0,y_1 \in \Lambda_1$ from 
Section \ref{SFermat} and the identification of $M$ with $\Lambda_1$ from Section~\ref{Shomology}.  
We consider
the following ordered basis $V_M$ of $M$:
\[\left[ 1, y_1, y_1^2, y_0, y_0y_1, y_0y_1^2, y_0^2, y_0^2y_1, y_0^2y_1^2 \right].\]
Let $B_\sigma, B_\tau \in \Lambda_1$ be such that $\sigma \cdot \beta = B_\sigma \beta$
and $\tau \cdot \beta = B_\tau \beta$. 
By \cite[Example~3.7]{WINF:Baction},
\[B_\sigma - 1 = y_0y_1 (1 - y_0 - y_1), \ B_\tau -1 = y_0y_1 ( - y_0 -y_1 + y_0y_1).\]
By \cite[Example 5.5(1)]{WINF:Baction}, when $p=3$ then
$\{y_0^2, y_0^2y_1, y_0^2y_1^2, y_0y_1^2, y_0^2y_1^2\}$ is a basis for $M^Q$.


\begin{lemma} \label{LbasisA} Let $p=3$ and $M=\rm{H}_1(U, Y; \mathbb{Z}/3 \mathbb{Z})$.
\begin{enumerate}
\item Then $\rm{H}^1(N, M)^Q = \rm{H}^1(N_7, M)^Q$ and $\dim_{\mathbb{F}_3}(\rm{H}^1(N_7, M)^Q)=18$.
\item There is a basis $\xi_1, \ldots, \xi_7$ for $N_7$ 
(also the images of $\{\xi_1, \ldots, \xi_\ell\}$ in $N_\ell$ are a basis for $N_\ell$ for $1 \leq \ell \leq 7$), 
such that $\rm{H}^1(N_7, M)^Q$ is spanned by the image of the 10-dimensional ${\rm{H}om}(N_2, M^Q)$ and the
8 maps $A_{11}, \ldots, A_{18}$:

$\begin{array}{cllll}
A_{11}:   & \xi_1 \mapsto y_1  & \xi_4 \mapsto  y_0y_1^2+y_0^2y_1^2 &  \xi_5 \mapsto y_0y_1^2 & \xi_7 \mapsto -y_0y_1^2-y_0^2y_1^2\\
A_{12}:   &   \xi_1 \mapsto y_0 & \xi_4\mapsto y_0^2y_1+y_0^2y_1^2 & \xi_5 \mapsto y_0^2y_1& \xi_7 \mapsto -y_0y_1^2-y_0^2y_1^2\\
A_{13} : &  \xi_1 \mapsto y_0y_1 & \xi_4 \mapsto y_0^2y_1^2 &  \xi_5 \mapsto y_0^2y_1^2 & \xi_7 \mapsto -y_0^2y_1^2\\
A_{14}  : &  \xi_3 \mapsto y_1^2 & \xi_4\mapsto -y_1^2 &  \xi_5 \mapsto y_1^2 & \xi_7 \mapsto y_1^2 \\
A_{15}  : &  \xi_3 \mapsto y_0y_1^2 & \xi_4\mapsto -y_0y_1^2 &  \xi_5 \mapsto y_0y_1^2 &  \xi_7 \mapsto y_0y_1^2\\
A_{16}  : &  \xi_3 \mapsto y_0^2  & \xi_4\mapsto -y_0^2 & \xi_5 \mapsto y_0^2 & \xi_7 \mapsto y_0^2\\
A_{17}  : &  \xi_3 \mapsto y_0^2y_1 & \xi_4\mapsto -y_0^2y_1 & \xi_5 \mapsto y_0^2y_1 & \xi_7 \mapsto y_0^2y_1 \\
A_{18}  : &  \xi_3 \mapsto y_0^2y_1^2  & \xi_4\mapsto -y_0^2y_1^2 & \xi_5 \mapsto y_0^2y_1^2 & \xi_7 \mapsto y_0^2y_1^2\\
\end{array}$

(All basis elements $\xi_i$ not listed map to $0$).

\item If $M_U=\rm{H}_1(U; \mathbb{Z}/3 \mathbb{Z})$, then $\rm{H}^1(N, M_U)^Q$ is spanned by 
the image of the 6-dimensional $\mathrm{Hom}(N_2, M_U^Q)$ and $\{A_{13}, A_{15}, A_{17}, A_{18}\}$; 
in particular, $\mathrm{dim}_{\mathbb{F}_3}(\rm{H}^1(N, M_U)^Q)=10$.

\item If $M_X=\rm{H}_1(X; \mathbb{Z}/3 \mathbb{Z})$, then $\rm{H}^1(N, M_X)^Q$ is spanned by the image of the $4$-dimensional
$\mathrm{Hom}(N_2, M_X^Q)$ and 
$\{A_{13}, A_{15}\}$; 
in particular, it follows that $\mathrm{dim}_{\mathbb{F}_3}(\rm{H}^1(N, M_X)^Q)=6$. 
\end{enumerate}
\end{lemma}

\begin{proof}
We prove each statement using a Magma calculation \cite{Magma}.  Here are some details for part (1).
By Corollary \ref{PrankN3}, $N$ is a ray class group with $N \simeq (\mathbb{Z}/3\mathbb{Z})^{10}$.
Using Magma, we find a basis for $N$ and $10 \times 10$ matrices $M_{\sigma,10}$ and $M_{\tau,10}$ 
for the conjugation action of $\sigma$ and $\tau$ on $N$ with respect to that basis.

Since $N$ acts trivially on $M$, an element of $\rm{H}^1(N,M)^Q$ can be uniquely represented as a $Q$-invariant homomorphism $\phi:N \to M$.
Since ${\rm dim}_{\mathbb{F}_3}(N) =10$ and ${\rm dim}_{\mathbb{F}_3}(M)=9$, 
$\phi$ is given by a $9 \times 10$ matrix $A_\phi$ 
with respect to the bases for $M$ and $N$. 
Then $\phi$ is $Q$-invariant if and only if, 
for every $\vec{n} \in N$,
\[A_\phi(\vec{n}^{\sigma})=B_{\sigma} \cdot A_\phi(\vec{n}), \ A_\phi(\vec{n}^{\tau})=B_{\tau} \cdot A_\phi(\vec{n}).\] 

To find the $Q$-invariant homomorphisms, we follow \cite[pages 21-22]{tensorbook} and set 
\[A_{\sigma,10}=M_{\sigma,10} \otimes I_9-I_{10} \otimes B_{\sigma}^t, \ A_{\tau,10}=M_{\tau,10} \otimes I_9 -I_{10}  \otimes B_{\tau}^t.\]
Then $\phi$ is $Q$-invariant if and only if $A_\phi \in \mathrm{Ker} (A_{\sigma,10}) \cap \mathrm{Ker}(A_{\tau,10})$.
Using Magma, we compute that ${\rm dim}(\mathrm{Ker} (A_{\sigma,10}) \cap \mathrm{Ker}(A_{\tau,10}))=18$.
Thus $\dim_{\mathbb{F}_p}(\rm{H}^1(N, M)^Q)=18$.

By an analogous calculation, $\dim_{\mathbb{F}_p}(\rm{H}^1(N_7, M)^Q)=18$.
By Lemma~\ref{LdecomposeH1}, the natural map $\rm{H}^1(N_7,M)^Q \rightarrow \rm{H}^1(N,M)^Q$ is injective.  
Thus $\rm{H}^1(N_7,M)^Q = \rm{H}^1(N,M)^Q$.
\end{proof}

\begin{remark} \label{R4kerd2}
Recall that $\bar{N}_H$ (resp.\ $N^*$) 
is the quotient of $N$ arising from the Heisenberg (resp.\ cyclotomic) extension of $K$.
By Propositions \ref{PCYC} and \ref{PsimplifyH1} and Notation \ref{NgenerateN},
$\rm{H}^1(N_2, M)^Q = \rm{H}^1(\bar{N}_H, M)^Q \oplus \rm{H}^1(N^*, M)^Q$.
Furthermore, by Propositions \ref{Pcyckerd2} and Example~\ref{Eheiskerd23}, 
${\rm Ker}(d_{2,H}) \simeq \rm{H}_1(U)^Q$ 
has dimension $3$ and
${\rm Ker}(d_2^*) \simeq \langle y_0^2y_1^2 \rangle$ has dimension $1$. 
\end{remark}

\begin{example}
Here are the matrices computed for the action of $\sigma$ and $\tau$ on $N_7$:
\begin{equation} \label{M7s}
M_{\sigma,7} = \left(\begin{array} {ccccccc}
1& 0& 1& 0& 1& 0& 1\\
0& 1 & 0& 0& 0& 2& 0 \\
0& 0 & 1& 2& 0& 2& 2 \\
0& 0 & 0 & 1 & 0 & 0 & 0 \\
0 & 0 & 0 & 0 & 1& 2 & 0 \\
0 & 0 & 0 & 0 & 0 & 1 & 0 \\
0 & 0 & 0 & 0 & 0 & 0 & 1 
\end{array} \right);
\ {\rm and} \ 
M_{\tau,7} = \left(\begin{array} {ccccccc}
1& 0& 0& 1& 2& 1& 2 \\
0& 1 & 0& 1& 0& 2& 1 \\
0& 0 & 1 & 0 & 0 & 1 & 0  \\
0 & 0 & 0 & 1 & 0& 0 & 0  \\
0 & 0 & 0 & 0 & 1 & 0 & 0 \\
0 & 0 & 0 & 0 & 0 & 1 & 0 \\
0 & 0 & 0 & 0 & 0 & 0 & 1 
\end{array} \right).\end{equation}
For $1 \leq \ell \leq 7$, $\sigma$ and $\tau$ act on $N_\ell$ 
via the upper-left $\ell \times \ell$ submatrices of $M_{\sigma,7}$ and $M_{\tau,7}$.
\end{example}

\subsection{Computing ${\rm Ker}(d_2)$ when $p=3$} \label{Skerd2}

By a Magma computation, ${\rm dim}_{\mathbb{F}_3}(\rm{H}^2(Q, N))=1$.
The extension in \eqref{Egaloisexact} therefore corresponds to an element of $\mathbb{F}^*_3$. 
We make an arbitrary choice of element of $\mathbb{F}^*_3$
and compute an explicit $2$-cocycle $\omega'$ representing $\omega \in \rm{H}^2(Q, N)$.
Since the elements of $\mathbb{F}^*_3$ negate each other, 
${\rm Ker}(d_2)$ is not affected by this choice.

By Section \ref{Sclass}, the element $\omega \in H^2(Q, N)$ classifying \eqref{Eshortexactsec4}
is determined by the elements $a,b,c \in N$ such that $a=s(\sigma)^3$, $b=s(\tau)^3$ and $c=[s(\sigma), s(\tau)]$.

\begin{proposition} \label{Pclass3}
Let $a_7, b_7, c_7$ denote the images of $a,b,c$ in $N_7$.
In terms of the basis $\xi_1, \ldots, \xi_7$,
\begin{equation} \label{Eabc}
a_7=[0, 2, 0, 2, 1, 0, 2], \ b_7=[0, 0, 0, 0, 0, 0, 2], \ {\rm and} \ c_7=[2,  1, 2,  0,  2,  1, 0].
\end{equation}
\end{proposition}

\begin{proof}
Using Magma, we compute a $2$-cocycle $\omega'$ for $\omega \in H^2(Q, N)$.
Using Lemma \ref{Lchangesection}, we compute $a= \omega'( \sigma, \sigma) + \omega'(\sigma^2, \sigma)$;
$b= \omega'(\tau, \tau) + \omega'( \tau^2, \tau)$; and
$c= \omega'(\tau, \sigma) - \omega'(\sigma, \tau)$.
\end{proof}

\begin{corollary} \label{Cfinish3} Let $p=3$.
\begin{enumerate}
\item If $M=\rm{H}_1(U,Y; \mathbb{Z}/3 \mathbb{Z})$ or $\rm{H}_1(U; \mathbb{Z}/3 \mathbb{Z})$, then ${\rm Ker}(d_2) = {\rm Ker}(d_{2,3})$ and ${\rm dim}_{\mathbb{F}_3}({\rm Ker}(d_{2}))=5$.
\item If $M=\rm{H}_1(X; \mathbb{Z}/3 \mathbb{Z})$, then ${\rm dim}_{\mathbb{F}_3}({\rm Ker}(d_2))=2$.
\end{enumerate}
\end{corollary}

\begin{proof}
By Lemma \ref{LbasisA}, a class $\phi \in \rm{H}^1(N,M)^Q$ 
is uniquely determined by a $Q$-invariant homomorphism $\bar{\phi}: N_7 \to M$.  
Also $\phi \in {\rm Ker}(d_2)$ if and only if $\bar{\phi} \in {\rm Ker}(\bar{d}_2)$ where 
$\bar{d}_2: \rm{H}^1(N_7, M)^Q \to \rm{H}^2(Q, M)$.  It thus suffices to 
compute using $H^1(N_7, M)^Q$, which is explicitly described in Lemma \ref{LbasisA}.

Let $\bar{\phi} \in \rm{H}^1(N_7, M)^Q$.
By Theorem \ref{Tkerd2}, 
when $p=3$, then $\bar{\phi}$ is in ${\rm Ker}(\bar{d}_2)$ if and only if there exist 
$m_0$ and $m_1$ in $M$ such that: 
\[\bar{\phi}(a_7) = - N_\sigma m_0, \  \bar{\phi}(b_7) = -N_\tau m_1, \ {\rm and} \   
\bar{\phi}(c_7) = (1-\tau)m_0 - (1-\sigma)m_1.\] 
Using Magma, this simplifies to $\bar{\phi}(a_7) \in \langle y_0^2 y_1^2 \rangle$ 
and $\bar{\phi}(b_7)=0$ and $\bar{\phi}(c_7) \in \rm{H}_1(U)^Q$. 

The calculations for $N_7$ and $N_3$ have the same outcome.
For $N_3$, we compute using Magma that
these conditions are satisfied if and only if 
\[\bar{\phi} \in {\rm Span}\{A_2, A_4, A_5, A_{10}, A_{13}+A_{18}\},\]
where $A_2: \xi_1 \mapsto y_0y_1^2$, $A_4: 
\xi_1 \mapsto y_0^2y_1$, $A_5: \xi_1 \mapsto y_0^2y_1^2$, 
$A_{10}: \xi_2 \mapsto y_0^2y_1^2$, and \[A_3+A_{18}: \xi_1 \mapsto y_0y_1, \  
\xi_3 \mapsto y_0^2y_1^2\] (all generators not listed map to zero).
Thus ${\rm dim}_{\mathbb{F}_3}({\rm Ker}(d_2))=5$.

Replacing $M$ by $\rm{H}_1(U)$, the images of the maps $A_2, A_4, A_5, A_{10}, A_{13}+A_{18}$ are in 
$\rm{H}_1(U)$ and ${\rm Ker}(d_2)$ is again 5-dimensional.
Next, we replace $M$ by $\rm{H}_1(X)$, which is 2-dimensional, say with basis $v_1$ and $v_2$.
In this case, ${\rm Ker}(d_2)$ has dimension 2 and a basis is given by $\phi_1: \xi_1 \mapsto v_1$ and $\phi_2: \xi_1 \mapsto v_2$, where all generators not listed map to zero.
\end{proof}

\begin{remark} \label{Ridentify}
When $M=\rm{H}_1(X; \mathbb{Z}/3 \mathbb{Z})$, this shows that ${\rm Ker}(d_2)$ is determined by the Heisenberg extension of $K$.
When $M = \rm{H}_1(U,Y; \mathbb{Z}/3 \mathbb{Z})$, then ${\rm Ker}(d_2)$ is determined by the extension $L_3/K$.
Note that $G_3={\rm Gal}(L_3/K)$ is non-abelian since $c_3=[2,1,2]$ is non-trivial and has exponent $9$ since
$a_3=[0,2,0]$ is non-trivial.  In Magma notation, the group $G_3$ is $\textbf{SmallGroup}(243, 13)$.
\end{remark}

\bibliography{raysubmit}
\bibliographystyle{plain}

\end{document}